\documentclass[11pt]{amsart}

\usepackage{amsmath,amsfonts,amssymb,amsthm}

\newtheorem{thm}{Theorem}[section]

\newtheorem{cor}[thm]{Corollary}

\newtheorem{lem}[thm]{Lemma}

\newtheorem{prop}[thm]{Proposition}

\newtheorem{example}[thm]{Example}

\newtheorem{remark}[thm]{Remark}

\newtheorem{defn}[thm]{Definition}

\usepackage{xcolor, soul}

\begin{document}
	
	\title{Strong jump inversion}
	
	\author{W.\ Calvert, A.\ Frolov, V.\ Harizanov, J.\ Knight,\\ 
		C.\ McCoy, A.\ Soskova, and S.\ Vatev}
	
	\address{Department of Mathematics\\
		Southern Illinois University\\
		USA}
	\email{wcalvert@siu.edu}
	
	\address{Department of Mathematics\\
		Kazan Federal University\\
		Russia}
	\email{Andrey.Frolov@ksu.ru}
	
	\address{Department of Mathematics\\
		George Washington University\\
		USA}
	\email{harizanv@gwu.edu}
	
	\address{Department of Mathematics\\
		University of Notre Dame\\
		USA}
	\email{knight.1@nd.edu}
	
	\address{Department of Mathematics\\
		University of Portland\\
		USA}
	\email{mccoy@up.edu}
	
	\address{Department of Mathematical Logic\\
		Sofia University\\
		Bulgaria}
	\email{asoskova@fmi.uni-sofia.bg}
	
	\address{Department of Mathematical Logic\\
		Sofia University\\
		Bulgaria}
	\email{stefanv@fmi.uni-sofia.bg}
	
	\thanks{The authors are grateful for support from NSF Grant DMS \#1101123, which allowed the first, third, fourth, and fifth authors to visit Sofia, allowed the sixth author to visit Notre Dame, and allowed the seventh author to visit University of Portland.  The second author is grateful for support from RFBR-16-31-60077 and from Kazan Federal University, which allowed him to visit Notre Dame.  The third author is grateful for the support from the Simons Foundation Collaboration Grant and from the CCFF and Dean's Research Chair Award of the George Washington University.  Finally, the sixth and seventh authors are grateful for support from BNSF, MON DN 02/16 and NSF Grant DMS \#1600625.  Finally, the authors are grateful to the referees, for helpful suggestions, and especially for making us aware of work of Marker and Miller on differentially closed fields.}
	
	\maketitle
	
	\begin{abstract}
		We say that a structure $\mathcal{A}$ admits \emph{strong jump inversion} provided that for every oracle $X$, if $X'$ computes $D(\mathcal{C})'$ for some $\mathcal{C}\cong\mathcal{A}$, then $X$ computes $D(\mathcal{B})$ for some $\mathcal{B}\cong\mathcal{A}$.  Jockusch and Soare \cite{JS} showed that there are low linear orderings without computable copies, but Downey and Jockusch \cite{DJ} showed that every Boolean algebra admits strong jump inversion.  More recently, D.\ Marker and R.\ Miller \cite{MM} have shown that all countable models of $DCF_0$ (the theory of differentially closed fields of characteristic $0$) admit strong jump inversion.  We establish a general result with sufficient conditions for a structure $\mathcal{A}$ to admit strong jump inversion.  Our conditions involve an enumeration of $B_1$-types, where these are made up of formulas that are Boolean combinations of existential formulas.  Our general result applies to some familiar kinds of structures, including some classes of linear orderings and trees.  We do not get the result of Downey and Jockusch for arbitrary Boolean algebras, but we do get a result for Boolean algebras with no $1$-atom, with some extra information on the complexity of the isomorphism.  Our general result gives the result of Marker and Miller.  In order to apply our general result, we produce a computable enumeration of the types realized in models of $DCF_0$.  This also yields the fact that the saturated model of $DCF_0$ has a decidable copy.    
	\end{abstract}
	
	\section{Introduction}
	
	We often identify a structure $\mathcal{A}$ with its atomic diagram $D(\mathcal{A})$. Then $D(\mathcal{A})'$ is the jump of the structure.  This is the only notion of jump that we shall actually use.  We are interested in the following notion of jump inversion.  
	
	\begin{defn}
		\label{maindefn}
		
		A structure $\mathcal{A}$ \emph{admits strong jump inversion} provided that for all sets $X$, if $X'$ computes $D(\mathcal{C})'$ for some $\mathcal{C}\cong\mathcal{A}$, then $X$ computes $D(\mathcal{B})$ for some 
		$\mathcal{B}\cong\mathcal{A}$.  
		
	\end{defn}
	
	\begin{remark}
		
		The structure $\mathcal{A}$ admits strong jump inversion iff for all $X$, if $\mathcal{A}$ has a copy that is low over $X$, then it has a copy that is computable in $X$.  Here when we say that $\mathcal{C}$ is low over $X$, we mean that $D(\mathcal{C})' \leq_T X'$.
		
	\end{remark}
	
	The definition of strong jump inversion was motivated by the following result of Downey and Jockusch \cite{DJ}. 
	
	\begin{thm} [Downey-Jockusch]
		\label{Downey-Jockusch}
		
		All Boolean algebras admit strong jump inversion.  
		
	\end{thm} 
	
	\begin{proof} [Sketch of proof]
		
		Let $\mathcal{A}$ be a Boolean algebra that is low over $X$.  Then $X'$ computes the set of atoms in 
		$\mathcal{A}$.  Downey and Jockusch showed that if $X'$ computes $(\mathcal{A},atom(x))$, then $X$ computes a copy of $\mathcal{A}$.  The proof involves some non-uniformity.  A Boolean algebra with only finitely many atoms obviously has a computable copy.  Suppose $\mathcal{A}$ has infinitely many atoms.  If 
		$\mathcal{A}$ is low over $X$, then there is an $X$-computable Boolean algebra $\mathcal{B}$ with a function $f$, $\Delta^0_2$ relative to $X$, which would be an isomorphism from $\mathcal{B}$ to $\mathcal{A}$ except that it may map a finite join of atoms in $\mathcal{B}$ to a single atom in $\mathcal{A}$.  We convert $f$ into an isomorphism by re-apportioning the atoms (see Vaught~\cite{Vaught}). 
	\end{proof} 
	
	Here are some further examples of structures that admit strong jump inversion.
	
	\begin{example} [Equivalence structures]
		
		Each equivalence structure is characterized up to isomorphism by the number of equivalence classes of various sizes.  We consider equivalence structures with infinitely many infinite classes.  It is well-known, and easy to prove, that such an equivalence structure has an $X$-computable copy iff the set of pairs $(n,k)$ such that there are at least $k$ classes of size $n$ is $\Sigma^0_2$ relative to $X$.  (See \cite{AK} for a complete characterization of the equivalence structures with computable copies.)  
		
		\begin{prop}
			\label{Equivalence}
			
			Let $\mathcal{A}$ be an equivalence structure with infinitely many infinite classes.  Then $\mathcal{A}$ admits strong jump inversion.
			
		\end{prop}
		
		\begin{proof}
			
			If $\mathcal{A}$ is low over $X$, then the set $Q$ consisting of pairs $(n,k)$ such that there are at least $k$ classes of size $n$ is $\Sigma^0_2$ relative to $\mathcal{A}$, so it is $\Sigma^0_2$ relative to $X$.  Then $\mathcal{A}$ has an $X$-computable copy.
		\end{proof}
		
	\end{example}
	
	\begin{example} [Abelian $p$-groups of length $\omega$]
		
	  By Ulm's Theorem, a countable Abelian $p$-group is characterized up to isomorphism by the Ulm sequence and the dimension of the divisible part.  For an account of this, see \cite{Kaplansky}.  An Abelian $p$-group of length $\omega$ can be expressed as a direct sum of copies of $Z_{p^{n+1}}$, for finite $n$, and the Pr\"{u}fer group $Z_{p^\infty}$.  Then the Ulm sequence is $(u_n(G))_{n\in\omega}$, where $u_n(G)$ is the number of direct summands of form $Z_{p^{n+1}}$.  The dimension of the divisible part is the number of direct summands of form $Z_{p^\infty}$.
          It is well-known, and easy to prove, that if $G$ is an Abelian $p$-group of length $\omega$ with a divisible part of infinite dimension, then $G$ has an $X$-computable copy iff the set $\{(n,k):u_n(G)\geq k\}$ is $\Sigma^0_2$ relative to $X$.              
		
		\begin{prop}
			\label{Groups}
			
			Let $G$ be an Abelian $p$-group of length $\omega$ such that the divisible part has infinite dimension.  Then $G$ admits strong jump inversion.
			
		\end{prop}
		
		\begin{proof} 
			
			Suppose $G$ itself is low over $X$.  The set $\{(n,k):u_n(G)\geq k\}$ is $\Sigma^0_2$ relative to $G$, so it is $\Sigma^0_2$ relative to $X$.  Then $G$ has an $X$-computable copy.          
		\end{proof}  
		
	\end{example}
	
	Not all countable structures admit strong jump inversion.  
	
	\begin{example}
		
		Jockusch and Soare \cite{JS} showed that there are low linear orderings with no computable copy.
		
	\end{example}
	
	\begin{example}
		
		Let $T$ be a low completion of $PA$.  There is a model $\mathcal{A}$ such that the atomic diagram $D(\mathcal{A})$, and even the complete diagram $D^c(\mathcal{A})$, are computable in $T$.  Then $D(\mathcal{A})'$ is $\Delta^0_2$.  By a well-known result of Tennenbaum, since $\mathcal{A}$ is necessarily non-standard, there is no computable copy.
		
	\end{example}
	
	In Section 2, we give a general result with sufficient conditions for strong jump inversion.  In Section 3, we give several applications of our general result.  The last of these gives the result of Marker and Miller \cite{MM} saying that all models of $DCF_0$ admit strong jump inversion.  We add a result saying that the countable saturated model of $DCF_0$ has a decidable copy.  In the remainder of Section 1, we mention briefly some further notions of jump and jump inversion.    
	
	\subsection{Further notions of jump and jump inversion}
	
	Recall that a relation $R$ is \emph{relatively intrinsically $\Sigma^0_\alpha$} on a structure $\mathcal{A}$ if in all (isomorphic) copies $\mathcal{B}$ of $\mathcal{A}$, the image of $R$ is $\Sigma^0_\alpha$ relative to $\mathcal{B}$.  By a result of \cite{AKMS} and \cite{C}, these are the relations that are definable in $\mathcal{A}$ by computable $\Sigma_\alpha$ formulas, with parameters.  
	
	There is a computable set of indices for computable $\Sigma_1$ formulas, so we can enumerate, uniformly in $D(\mathcal{A})$, all relations that are relatively intrinsically $\Sigma^0_1$ on $\mathcal{A}$ (\emph{r.i.c.e.}).  Moreover, we can uniformly compute all of these relations from the Turing jump of the diagram, $D(\mathcal{A})'$.  The \emph{jump} of $\mathcal{A}$ is often defined to be a structure $\mathcal{A}'$ obtained by adding to $\mathcal{A}$ a specific named family of r.i.c.e.\ relations, from which all others are effectively obtained.  Then the r.i.c.e. relations on $\mathcal{A}'$ are just those which are relatively intrinsically $\Sigma^0_2$ on $\mathcal{A}$ itself.
	
	\begin{defn} [Canonical jump]
		
		For a structure $\mathcal{A}$, the \emph{canonical jump} is a structure $\mathcal{A}' = (\mathcal{A},(R_i)_{i\in\omega})$, where $(R_i)_{i\in\omega}$  are relations from which we can uniformly compute all r.i.c.e.\ relations on $\mathcal{A}$, and from the index $i$ of the relation $R_i$, we can compute the arity of $R_i$ and a computable $\Sigma_1$ formula (without parameters) that defines it in $\mathcal{A}$.  
	\end{defn}
	
	\begin{remark}
		\label{rem:1}
		The set $\emptyset'$ is included in the canonical jump.  We may give it by a family of relations $R_{f(e)}$, for a computable function $f$, where $R_{f(e)}$ is always true if $e\in\emptyset'$ and always false otherwise.  We may define $R_{f(e)}$ by the computable $\Sigma_1$ formula $\bigvee_s \tau_{e,s}$, where $\tau_{e,s}$ is $\top$ if $e$ has entered $\emptyset'$ by step $s$ and $\bot$ otherwise.  
	\end{remark} 
	
	The original definition of the \emph{jump} of $\mathcal{A}$, as a structure, appears in the Ph.D.\ thesis of Baleva, supervised by Soskov \cite{BalevaThesis,B}.  The definition was later used by  A.~Soskova and Soskov \cite{SS} for some jump inversion theorems.  The definition in \cite{SS} looks slightly different.  Some arithmetic is added to the structure, and the sequence of relations is coded by a single relation.  The domain of $\mathcal{A}'$ is the ``Moschovakis extension'' of $\mathcal{A}$ with an appropriate coding mechanism, and the added relation is one that codes the forcing relation of the computable infinitary $\Sigma_1^0$ formulas as an analogue of Kleene's set $K$.  There is yet another notion of jump, which involves $\Sigma$-definability in the hereditarily finite sets over a base structure.  This notion appears in work of Morozov \cite{Morozov}, Stukachev \cite{S1,Stukachev}, Puzarenko \cite{P}, and others from the Novosibirsk school.  It applies to base structures of arbitrary cardinality.  
	
	Montalb\'{a}n \cite{Montalban} initially used relatively intrinsically $\Pi^0_1$ relations instead of r.i.c.e.\ relations.  The definition given above is a modification of the one in \cite{Montalban}, which was arrived at after some group discussions in Sofia in the summer of 2011.  In the spring of 2012, Russian, Bulgarian, and U.S.\ researchers gathered in Chicago for further discussions of the notions of jump, at the workshop ``Definability in computable structures'', funded mainly by the Packard Foundation.  Later, Montalb\'{a}n proved that the different-looking definitions are equivalent (see \cite{Montalban2}). 
	
	For some structures, there is a smaller subset of the relations that is sufficiently powerful to replace the full set.
	
	\begin{defn}[Structural jump]
		A \emph{structural jump} of $\mathcal{A}$ is an expansion $\mathcal{A}' = (\mathcal{A},(R_i)_{i\in\omega})$ such that each $R_i$ has a $\Sigma_1$ defining formula that we can compute from $i$, and every relation that is relatively intrinsically $\Sigma^0_2$ on $\mathcal{A}$ is r.i.c.e.\ on $\mathcal{A}'\oplus\emptyset'$. 
	\end{defn}
	Here the structure $\mathcal{A}'\oplus\emptyset'$ is the expansion of $\mathcal{A}'$ by a family of relations that encode the set $\emptyset'$, as explained in Remark \ref{rem:1}.
	
	For certain classes of structures, there is a structural jump formed by adding a finite set of such relations.  In particular, the relation $atom(x)$ is sufficient for Boolean algebras, and the successor relation $succ(x,y)$ is sufficient for linear orders.  See \cite{Montalban} for further examples.
	
	\bigskip
	
	There are different statements of ``jump inversion''.  The well-known Friedberg Jump Inversion Theorem says that if $\emptyset'\leq_T Y$, then there is a set $X$ such that $X'\equiv_T Y\equiv \emptyset'\oplus X$.  We can easily produce a structure $\mathcal{B}$ such that $X\equiv_T\mathcal{B}$, and then $Y\equiv_T\mathcal{B}'$.  This is one kind of jump inversion.  A more interesting kind of jump inversion theorem was proved by Soskov and A.\ Soskova \cite{ASJ}, \cite{SS}, and later (independently) by Montalb\'{a}n \cite{Montalban}.
	
	\begin{thm} [Soskov, A.\ Soskova, Montalb\'{a}n]\label{JIT}
		
		For any countable structure $\mathcal{A}$, if $Y$ computes a copy of the canonical jump $\mathcal{A}'$ of 
		$\mathcal{A}$, there exists a set $X$ such that $X'\equiv_T Y$ and $X$ computes a copy of 
		$\mathcal{A}$.
		
	\end{thm}
	
	The proposition below shows that we can express strong jump inversion in terms of copies of the canonical jump structure $\mathcal{A}'$, as opposed to the Turing jump of the atomic diagram for various copies 
	$\mathcal{B}$.
	
	\begin{prop} For any structure $\mathcal{A}$, the following are equivalent:
		
		\begin{enumerate}
			
			\item $\mathcal{A}$ admits strong jump inversion.   
			
			\item For all sets $X$, if $X'$ computes a copy of the canonical jump $\mathcal{A}'$ of $\mathcal{A}$, then $X$ computes a copy of $\mathcal{A}$.
			
			\item  For all sets $X$ and $Y$, if $X'\equiv_T Y'$ and $Y$ computes a copy of $\mathcal{A}$ then so does  $X$.
			
		\end{enumerate}
		
	\end{prop}
	
	\begin{proof} 
		
		For $(2) \Rightarrow (1)$, assume $\mathcal{A}$ has a copy $\mathcal{B}$ with $(D(\mathcal{B}))' \leq_T X'$.  Since $D(\mathcal{B}')\leq_T (D(\mathcal{B}))'\leq_T X'$, $(2)$ implies that $X$ computes a copy of $\mathcal{A}$.  
		
		For $(1) \Rightarrow (3)$, let $X'\equiv_T Y'$, where $Y$ computes a copy $\mathcal B$ of $\mathcal A$. Then $X'$ computes $D(\mathcal B)'$. By $(1)$, $X$ computes a copy of $\mathcal{A}$.
		
		For $(3) \Rightarrow (2)$, suppose $X'$ computes a copy of $\mathcal{A}'$. By Theorem~\ref{JIT}, there exists $Y$ such that $Y$ computes a copy of $\mathcal{A}$ and $Y'\equiv_T X'$.  By $(3)$, $X$ computes a copy of $\mathcal{A}$.
	\end{proof} 
	
	\section{General result}
	\label{general}
	
	In this section, we give a result with conditions sufficient to guarantee that a structure admits strong jump inversion.  The result is not difficult to prove.  However, there are a number of examples where it applies. To state the result, we need some definitions.
	
	\begin{defn}
		
		Let $S$ be a countable family of sets.  An \emph{enumeration} of $S$ is a set $R$ of pairs $(i,k)$ such that $S$ is the family of sets $R_i = \{k:(i,k)\in R\}$.  If $A = R_i$, we say that $i$ is an \emph{$R$-index} for $A$.
		
	\end{defn}
	
	\noindent
	\textbf{Note}:  When we say that $R$ is a \emph{computable enumeration} of a family of sets, we mean that $R$ is a computable set of pairs.  This means that the sets $R_i$ are \emph{computable}, uniformly in $i$.  Some researchers have used the term differently, saying that $R$ is a \emph{computable enumeration} if the sets $R_i$ are uniformly \emph{computably enumerable}. 
	
	\bigskip
	
	Below, we define $B_n$-types precisely.  We shall focus on $B_1$-types.  
	
	\begin{defn}\
		
		\begin{enumerate}
			
			\item  A \emph{$B_n$-formula} is a finite Boolean combination of ordinary finite elementary $\Sigma_n$-formulas.
			
			\item  A \emph{$B_n$-type} is the set of $B_n$-formulas in the complete type of some tuple in some structure for the language.  
			
		\end{enumerate}
		
	\end{defn} 
	
	\begin{defn}
		
		Fix a structure $\mathcal{A}$.  Let $S$ be a set of $B_1$-types including all those realized in $\mathcal{A}$.  Let $R$ be an enumeration of $S$.  An \emph{$R$-labeling of $\mathcal{A}$} is a function taking each tuple $\bar{a}$ in $\mathcal{A}$ to an $R$-index for the $B_1$-type of $\bar{a}$.
		
	\end{defn}
	
	We are interested in structures $\mathcal{A}$ with the following property.  
	
	\begin{defn} [Effective type completion]
		
		The structure $\mathcal{A}$ satisfies \emph{effective type completion} if there is a uniform effective procedure that, given a $B_1$-type $p(\bar{u})$ realized in $\mathcal{A}$ and an existential formula $\varphi(\bar{u},x)$ such that $(\exists x)\varphi(\bar{u},x)\in p(\bar{u})$, yields a $B_1$-type $q(\bar{u},x)$ with $\varphi(\bar{u}, x) \in q(\bar{u}, x)$, such that if $\bar{a}$ in $\mathcal{A}$ realizes $p(\bar{u})$, then some $b$ in $\mathcal{A}$ realizes $q(\bar{a},x)$.     
		
	\end{defn}
	
	Here is our general result.
	
	\begin{thm}
		\label{simple.general}
		
		A structure $\mathcal{A}$ admits strong jump inversion if it satisfies the following conditions:
		
		\begin{enumerate}
			
			\item  There is a computable enumeration $R$ of a set of $B_1$-types including all those realized by tuples in 
			$\mathcal{A}$.
			
			\item  $\mathcal{A}$ satisfies effective type completion.
			
			\item  For all sets $X$, if $X'$ computes the jump of some copy of $\mathcal{A}$, then $X'$ computes a copy of $\mathcal{A}$ with an $R$-labeling. 
			
		\end{enumerate}
		Moreover, if $\mathcal{C}$ is a copy of $\mathcal{A}$ with an $X'$-computable $R$-labeling, then we get an $X$-computable copy $\mathcal{B}$ of $\mathcal{A}$ with an $X'$-computable isomorphism from $\mathcal{B}$ to $\mathcal{C}$.     
		
	\end{thm}
	
	\begin{remark}
		\label{SimpleCondition3}
		
		For some structures $\mathcal{A}$, Condition (3) is satisfied in a strong way.  For any $\mathcal{C}\cong\mathcal{A}$, $D(\mathcal{C})'$ computes an $R$-labeling of $\mathcal{C}$.  Hence, if $\mathcal{A}$ is low, there is a $\Delta^0_2$ isomorphism from $\mathcal{A}$ to a computable copy. 
		
	\end{remark}          
	
	\begin{proof} [Proof of Theorem \ref{simple.general}]
		
		Suppose that $\mathcal{A}$ satisfies the three conditions.  Let $X$ be a set such that $X'$ computes the jump of some copy of $\mathcal{A}$.  By Condition (3), $X'$ computes a copy with an $R$-labeling.  We must show that there is an $X$-computable copy.  For simplicity, we suppose that $\mathcal{A}$ has a $\Delta^0_2$ $R$-labeling, and we produce a computable copy $\mathcal{B}$, basing our construction on guesses at various portions of the $R$-labeling of $\mathcal{A}$.  Note that once we have guessed the label for a tuple $\bar{a}$ correctly, we computably know the entire $B_1$-type of that tuple.  We build a computable copy $\mathcal{B}$ and a $\Delta^0_2$ isomorphism $f$ from $\mathcal{B}$ to $\mathcal{A}$.  We have the following requirements. 
		
		\bigskip
		\noindent
		\textbf{$R_{2a}$}:  $a\in ran(f)$
		
		\bigskip
		\noindent
		\textbf{$R_{2b+1}$}:  $b\in dom(f)$
		
		\bigskip
		
		We start with an $R$-index for the type of $\emptyset$, where this type is the $B_1$-theory of $\mathcal{A}$.  At each stage $s$, we have a tentative partial isomorphism $f_s$ mapping a tuple $\bar{d}$ from $\mathcal{B}$ to a tuple $\bar{c}$ in $\mathcal{A}$, where the $R$-indices of the types of $\bar{c}$ and all of its initial segments still look correct.  (At a later stage $t$, we may see that some of the guesses at these indices are incorrect, and we retain only the portion of $f_s$ satisfying an initial segment of requirements based on guesses at $R$-indices that all look correct.)  Moreover, we have enumerated a finite part $\delta(\bar{d},\bar{b})$ of the atomic diagram of
		$\mathcal{B}$; this can never change, since $\mathcal{B}$ must be computable.  We will have checked the consistency of $\delta(\bar{d},\bar{b})$  with our guesses at the $R$-indices of the $B_1$-types of the tuple $\bar{c}$ and its initial segments.  Supposing that the function taking $\bar{d}$ to $\bar{c}$ satisfies the earlier requirements, we can satisfy the requirement $R_{2a}$ once we guess the $R$-index for the $B_1$-type $p(\bar{u},x)$ of $\bar{c},a$.  We map some $b$, either old or new, to $a$ so that $\delta(\bar{u},\bar{v})$ is consistent with $p(\bar{u},x)$.  (Recall that the $B_1$-types are computable.)
		
		Suppose that the function taking $\bar{d}$ to $\bar{c}$ satisfies the requirement $R_i$ for all $i < 2b+1$, and $R_{2b + 1}$ is least that is unsatisfied at this stage $s$.  Again, we assume that we have correct guesses on the $R$-indices for the $B_1$ types of $\bar{c}$ and all of its initial segments; let $p(\bar{u})$ be the $B_1$-type of $\bar{c}$.  Finally, we have put $\delta(\bar{d},b,\bar{b})$ in the atomic diagram of $\mathcal{B}$.  Now we use the assumption of effective type completion.  We determine, effectively in $p(\bar{u})$ and the existential formula $(\exists\bar{v})\delta(\bar{u},x,\bar{v})$, a type $q(\bar{u},x)$ appropriate for $\bar{c}$ and a putative $f_s(b)$.  If $\bar{c}$ realizes $p(\bar{u})$, then some $a$ will realize $q(\bar{c},x)$.  At step $s$, we can give a computable index for $q(\bar{u},x)$, but not an $R$-index.   
		
		By effective type completion, if $p(\bar{u})$ really is the $B_1$-type of $\bar{c}$, then $q(\bar{c},x)$ will be realized in $\mathcal{A}$.  We define $f_s(b)$ as follows.  We find the first $a$ such that, based on our guess at the $R$-index of the $B_1$ type of $\bar{c}, a$, this type and $q(\bar{u}, x)$ agree on the first $s$ formulas; then $f_s(b) = a$.  Of course, this guess at the element $a$ is likely wrong.  Therefore, in order to guarantee that this requirement is satisfied, at each subsequent stage $t$, we need to check that, based on our guess at the $R$-index of the $B_1$ type of $\bar{c}, a$, this type and $q(\bar{u}, x)$ agree on the first $t$ formulas.  If this is not the case, then we need to re-define $f_t(b)$, but always maintaining $q(\bar{u}, x)$ as the guaranteed type of $q(\bar{c}, f(b))$, so long as our work on earlier requirements seems correct.  (In particular, note that as we check consistency of the atomic diagram with the $B_1$ types associated with requirement $R_{2b+1}$, we use the computable index for $q(\bar{u}, x)$.)   
		There is a first $a$ realizing $q(\bar{c}, x)$, and eventually, we will have the $R$-index for the $B_1$ type of 
		$\bar{c},a$.  Then we will have $f_s(b) = f(b) = a$.   
	\end{proof}
	
	In several examples, $\mathcal{A}$ has effective type completion because it satisfies a property that we call \emph{weak $1$-saturation}.  To describe this property, we need a preliminary definition.     
	
	\begin{defn}
		
		Suppose $p(\bar{u})$ and $q(\bar{u},x)$ are $B_1$-types.  We say that $q(\bar{u},x)$ is \emph{generated by the formulas of $p(\bar{u})$ and existential formulas} provided that $q(\bar{u},x)\supseteq p(\bar{u})$, and for any universal formula $\psi(\bar{u},x)$ (in the indicated variables), writing $neg(\psi)$ for the natural existential formula logically equivalent to $\neg{\psi}$, we have $\psi(\bar{u},x)\in q(\bar{u},x)$ iff there is a finite conjunction $\chi(\bar{u},x)$ of existential formulas in $q(\bar{u},x)$ such that $(\exists x)[\chi(\bar{u},x)\ \&\ neg(\psi(\bar{u},x))]$ is not in $p(\bar{u})$.
		
	\end{defn}
	
	\begin{defn}
		
		The structure $\mathcal{A}$ is \emph{weakly $1$-saturated} provided that if $p(\bar{u})$ is the $B_1$-type of a tuple $\bar{a}$, and $q(\bar{u},x)$ is a $B_1$-type generated by formulas of $p(\bar{u})$ and existential formulas, then $q(\bar{a},x)$ is realized in $\mathcal{A}$.
		
	\end{defn}
	
	The following is clear.  
	
	\begin{lem}
		\label{prop:existential-types}
		
		Let $p(\bar{u})$ be a $B_1$-type.  Suppose $q(\bar{u},x)$ is a $B_1$-type that is generated by formulas of $p(\bar{u})$ and existential formulas.  Then $q(\bar{u},x)$ is consistent with all extensions of $p(\bar{u})$ to a complete type in variables $\bar{u}$.
		
	\end{lem}
	
	\begin{prop}
		
		If $\mathcal{A}$ is weakly $1$-saturated, then it satisfies effective type completion.
		
	\end{prop}
	
	\begin{proof}
		
		Let $p(\bar{u})$ be a $B_1$-type, and suppose $\varphi(\bar{u},x)$ is an existential formula such that $(\exists x)\varphi(\bar{u},x)\in p(\bar{u})$.  We effectively produce a type $q(\bar{u},x)$ extending $p(\bar{u})$ and containing the formula $\varphi(\bar{u},x)$, such that if $\bar{a}$ realizes $p(\bar{u})$, then some $b$ realizes $q(\bar{a},x)$.  The type $q(\bar{u},x)$ is generated by formulas of $p(\bar{u})$ and existential formulas, including the formula $\varphi(\bar{u},x)$.  We determine this $B_1$-type computably as follows.  We start with $p(\bar{u})$ and $\varphi(\bar{u},x)$.  We have a computable list $(\varphi_n(\bar{u},x))_{n\in\omega}$ of all existential formulas in variables $\bar{u},x$, in order of G\"{o}del number.  We consider these formulas, in order, and we put $\varphi_n(\bar{u},x)$ into $q(\bar{u},x)$ iff it is consistent with what we have already put into $q(\bar{u},x)$.  (This consistency check is computable relative to $p(\bar{u})$, because it entails only asking whether the relevant $B_1$ formulas are in $p(\bar{u})$.)  If we fail to put $\varphi_n(\bar{u},x)$ in, then all tuples satisfying what we did put in must satisfy $neg(\varphi_n(\bar{u},x))$, so that is in $q(\bar{u},x)$.  Knowing exactly which existential formulas are in $q(\bar{u},x)$, we can determine which $B_1$ formulas are in (using truth tables).  We have described an effective procedure for determining $q(\bar{u},x)$.  By weak $1$-saturation, there is some $b$ in $\mathcal{A}$ realizing $q(\bar{a},x)$.    
	\end{proof}
	
	\section{Examples}
	\label{examples}
	
	In this section, we consider some examples of structures that admit strong jump inversion.  The examples are chosen to illustrate the use of Theorem \ref{simple.general}.  In Subsection \ref{linear orderings}, we discuss two special kinds of linear orderings.  For both, we can apply Theorem \ref{simple.general}.  For the first, Condition (3) holds in a strong way, as in Remark \ref{SimpleCondition3}.  In Subsection \ref{Boolean algebras}, we consider Boolean algebras with no $1$-atoms.  The result of Downey and Jockusch says that every low Boolean algebra has a computable copy.  In the case where there are no $1$-atoms, our result gives a $\Delta^0_3$ isomorphism from a low copy to a computable one.  In Subsection \ref{trees}, we apply Theorem \ref{simple.general} to some 
	special classes of trees.  
	
	In Subsection \ref{small}, we consider models of an $\aleph_0$-categorical elementary first order theory $T$ such that $T\cap\Sigma_2$ is computably enumerable.  The fact that the $B_1$-types are all isolated makes it easy to produce a computable enumeration.  By contrast, in Subsection \ref{DCF_0}, we consider models of the theory of differentially closed fields of characteristic $0$.  Here, although the theory is decidable, with all types computable, producing a computable enumeration of them is not trivial.  We get a result of Marker and R.\ Miller \cite{MM} saying that all models of $DCF_0$ admit strong jump inversion.  Moreover, a result of Morley in \cite{Morley} implies that, since the types of the theory have a computable enumeration, the saturated model of $DCF_0$ has a decidable copy.         
	
	\subsection{Linear orderings}
	\label{linear orderings}
	
	The second author proved strong jump inversion for two special classes of linear orderings, with further results on complexity of isomorphisms.  The results are given in \cite{F1}, \cite{F2}, \cite{F}.  Here we  prove these results using Theorem~\ref{simple.general}.
	
	First, we describe the possible $B_1$ types in linear orderings.  Every $B_1$-type $p(\bar{u})$ is determined uniquely by the sizes of the intervals to the left of the first element, between successive elements, and to the right of the last element.  Thus, we can define a computable enumeration $R$ of all $B_1$-types realized in linear orderings so that from the index $i$ of the $B_1$-type $R_i$, we can effectively obtain the sizes of the intervals.
	
	Let $p(u_1,u_2)$ be a $B_1$-type in which the interval $(u_1,u_2)$ is infinite.  We consider $B_1$-types $q(u_1,u_2,x)$, with $u_1 < x < u_2$.  To understand which of these are generated by formulas from $p(u_1,u_2)$ and existential formulas, it is helpful to consider the following cases.  
	
	\bigskip
	\noindent
	\textbf{Case 1}:   Let $q(u_1,u_2,x)$ be a $B_1$-type such that the interval $(u_1,x)$ is finite, of size $k$, and the interval $(x,u_2)$ is infinite.  Let $t(u_1,u_2)$ be a complete type saying that $u_1$ and $u_2$ are infinitely far apart and $u_1$ belongs to a maximal discrete interval of size less than $k$.
	Clearly, $p(u_1,u_2)$ is consistent with $t(u_1,u_2)$, whereas $q(u_1,u_2,x)$ is not.
	By Lemma~\ref{prop:existential-types}, $q(u_1,u_2,x)$ is not generated by $p(u_1,u_2)$ and existential formulas.
	
	\bigskip
	\noindent
	\textbf{Case 2}:  Let $q(u_1,u_2,x)$ extend $p(u_1,u_2)$ such that $u_1 < x < u_2$, and the intervals $(u_1,x)$ and $(x,u_2)$ are both infinite.  Then $q(u_1,u_2,x)$ is generated by $p(u_1,u_2)$ and the infinite set of existential formulas saying that for each $n$, there are at least $n$ elements in the intervals $(u_1,x)$ and $(x,u_2)$.
	
	\begin{prop}
		Let $\mathcal{A}$ be a linear ordering such that every infinite interval can be split into two infinite parts.
		Then $\mathcal{A}$ is weakly $1$-saturated.
	\end{prop}
	\begin{proof}
		For a tuple $\bar{a}$, we consider the possible $B_1$-types $q(\bar{a},x)$.  First, suppose $q(\bar{a},x)$ locates $x$ in a finite interval $(-\infty,a_0)$, $(a_i,a_{i+1})$, or $(a_n,\infty)$ so that the sizes of the two subintervals to the left and right of $x$ add up properly.  Then $q(\bar{u},x)$ is generated by formulas of $p(\bar{u})$ and existential formulas saying that the subintervals have at least the desired size, and $q(\bar{a},x)$ must be realized.  Next, suppose $q(\bar{a},x)$ locates $x$ in an infinite interval $(-\infty,a_0)$, $(a_i,a_{i+1})$, or $(a_n,\infty)$.  If $q(\bar{a},x)$ is generated by formulas of $p(\bar{u})$ and existential formulas, then $x$ must split the interval into two infinite parts.  The ordering $\mathcal{A}$ has exactly this feature.
	\end{proof}
	
	Here is the simpler of the two results on linear orderings.  
	
	\begin{thm} 
		\label{OrderingsBound}
		
		Let $\mathcal{A}$ be a linear ordering such that each element lies on a maximal discrete set that is finite.  Suppose there is a finite bound  on the sizes of these sets.  Then $\mathcal{A}$ admits strong jump inversion.  Moreover, if $\mathcal{A}$ is low over $X$, then there is an $X$-computable copy with an isomorphism that is $\Delta^0_2$ relative to $X$.  
		
	\end{thm}
	
	\begin{proof}
		
		Let $N$ be the finite bound on the sizes of the maximal discrete sets.  It is $\Delta^0_2$ relative to $\mathcal{A}$ to say that the interval $(a,b)$ has size $n$ for some fixed $n$.  It is $\Sigma^0_1$ relative to $\mathcal{A}$ to say that the interval is infinite---we just ask whether the interval has size greater than $N$.
		
		Suppose that $\mathcal{A}$ is low over $X$.  We can apply a procedure that is $\Delta^0_2$ relative to $X$ to assign an $R$-index to the type of any tuple $\bar{a} = (a_1,\ldots,a_n)$.  Any of the intervals 
		$(-\infty,a_1)$, $(a_n,\infty)$ and $(a_i,a_{i+1})$ is infinite if it has size greater than $N$. Using a procedure that is $\Delta^0_2$ relative to $X$, we can determine whether the size is $k$, for $k\leq N$.  We have an $R$-labeling of $\mathcal{A}$ that is $\Delta^0_2$ relative to $X$.  Then Theorem~\ref{simple.general} gives an $X$-computable copy with an isomorphism that is $\Delta^0_2$ relative to $X$.
	\end{proof}
	
	The next result, Theorem \ref{OrderingsNoBound}, is more complicated.  Before we state the result, we review some well-known, basic concepts about linear orderings.  Recall the \emph{block equivalence relation} $\sim$ on a linear ordering $\mathcal{A}$, where $a \sim b$ iff $[a, b]$ is finite.  For any linear ordering $\mathcal{A}$, each equivalence class under this relation is an interval that is either finite or of order type $\omega, \omega^*,$ or $\zeta = \omega^* + \omega$.  Furthermore, the quotient structure $\mathcal{A}/_{\sim}$ is itself a linear ordering, where each distinct point represents an equivalence class under $\sim$.      
	
	In Theorem \ref{OrderingsNoBound}, for a given $\mathcal{A}$ that is low over $X$, it is not clear that 
	$\mathcal{A}$ itself has an $R$-labeling that is $\Delta^0_2$ relative to $X$.  However, we can build a copy $\mathcal{B}$ with such an $R$-labeling.  We write $\eta$ for the order type of the rationals.         
	
	\begin{thm} 
		\label{OrderingsNoBound}
		
		Let $\mathcal{A}$ be a linear ordering for which the quotient $\mathcal{A}/_{\sim}$ has order type 
		$\eta$.  Suppose also that in $\mathcal{A}$, every infinite interval has arbitrarily large finite successor chains.  Then $\mathcal{A}$ admits strong jump inversion.  Moreover, if $\mathcal{A}$ is low over $X$, then there is an $X$-computable copy $\mathcal{B}$ with an isomorphism that is $\Delta^0_3$ over $X$ from $\mathcal{A}$ to $\mathcal{B}$.  
		
	\end{thm}
	
	\begin{proof}
		
		As in the previous result, let $R$ be a computable enumeration of all $B_1$-types realized in linear orderings, such that from the index $i$ of the type $R_i$, we can compute the sizes, including 
		$\infty$, of the intervals.  Also, as in the previous result, every infinite interval in $\mathcal{A}$ has an element that splits the interval into two infinite parts.  This implies that $\mathcal{A}$ is weakly $1$-saturated.  Suppose $\mathcal{A}$ is low over $X$.  We will prove the following.
		
		\begin{lem}
			
			There is a copy $\mathcal{B}$ of $\mathcal{A}$ with an $R$-labeling that is $\Delta^0_2$ over $X$.  Moreover, there is an isomorphism $f$ from $\mathcal{B}$ to $\mathcal{A}$ such that $f$ is $\Delta^0_3$ relative to $X$.    
			
		\end{lem}
		
		Assuming the lemma, we complete the proof of Theorem~\ref{OrderingsNoBound} as follows.  Given 
		$\mathcal{A}$, low over $X$, the lemma gives a copy $\mathcal{B}$ with an $R$-labeling that is 
		$\Delta^0_2$ relative to $X$, and an isomorphism $f$ from $\mathcal{B}$ to $\mathcal{A}$ that is 
		$\Delta^0_3$ relative to $X$.  By Theorem~\ref{simple.general}, there is an $X$-computable copy 
		$\mathcal{C}$ with an isomorphism $g$ from $\mathcal{C}$ to $\mathcal{B}$ that is $\Delta^0_2$ relative to $X$.  Then $f\circ g$ is an isomorphism from $\mathcal{C}$ to $\mathcal{A}$ that is $\Delta^0_3$ relative to $X$.

		\begin{proof} [Proof of Lemma]
			
			For simplicity, we suppose that $\mathcal{A}$ is low.  We build a $\Delta^0_2$ copy $\mathcal{B}$, along with some labels for sizes of intervals and a $\Delta^0_3$ isomorphism $f$.  We suppose that the universe of $\mathcal{A}$ is $\omega$.  The copy $\mathcal{B}$, also with universe $\omega$, will have the intervals labeled by size.  Throughout, we use the oracle $\Delta^0_2$.  Suppose $\mathcal{A}_n$ is the true ordering on the first $n$ elements of $\mathcal{A}$, with the intervals correctly labeled by size.  At stage $s$, we construct (using the $\Delta^0_2$ oracle) an approximation $\mathcal{A}_{n,s}$ in which the intervals are either correctly labeled with a finite number at most $s$, or else carry the label $\infty$.  We have a finite sub-ordering $\mathcal{B}_s$ of $\mathcal{B}$ in which the intervals are labeled by size, once and for all.  
			
			We want an isomorphism $f$ from $\mathcal{B}$ onto $\mathcal{A}$.  We must satisfy the following requirements.
			
			\bigskip
			\noindent
			\textbf{$R_{2a}$}:  Put $a$ into $ran(f)$.
			
			\bigskip
			\noindent
			\textbf{$R_{2b+1}$}:  Put $b$ into $dom(f)$.
			
			\bigskip
			
		By the end of each stage $s$, we have a finite function $f_s$ that seems to satisfy the first few requirements, so that our current labels on the intervals with endpoints in $ran(f_s)$ match the labels on the corresponding intervals in $dom(f_s)$.  Moreover, we ensure that if $f_s(b) = a$, then for any successor chain around $b$ in $\mathcal{B}_s$, we also have seen, by stage $s$, a corresponding successor chain around $a$ in $\mathcal{A}$.
		
		An interval that seemed infinite at stage $s$ may be seen to be finite at stage $s+1$.  So in defining $f_{s+1}$, we first determine the largest initial segment of $f_s$ (in terms of priority requirements) that can be preserved.  Consider the highest priority requirement that now must be satisfied.          
			
			Suppose the next requirement to be satisfied is to put $a$ into $ran(f)$.  We have no problem finding an appropriate pre-image $b$
			and assigning the appropriate sizes to the intervals having $b$ as an endpoint.
			
			Suppose the next requirement to be satisfied is to put $b$ into $dom(f)$.  In the interesting sub-case, $b$ lies in an interval $(d,d')$, where $(d,b)$ and $(b,d')$ are both labeled infinite, $f_s(d) = c$ and $f_s(d') = c'$, where $(c,c')$ appears to be infinite.  We need to define $a = f_{s+1}(b)$ such that $(c,a)$ and $(a,c')$ both appear infinite, and whatever successor chain surrounding $b$ is matched by one surrounding $a$.  The naive strategy is to just look for $a$.  This strategy may not work.  Believing that we have found $a$, and seeing that $a$ lies in a finite interval inside $(c,c')$, we may create a bigger successor chain around $b$, inside $(d,d')$.  Eventually, we may discover that the interval $(c,a)$ or $(a, c')$ is finite.  Now, we cannot map $b$ to $a$.  Moreover, we have made the search for $f(b)$ more difficult in that it must lie in a larger finite interval matching the one we have created around $b$.  This can keep happening.  Our current guess at the appropriate $a = f(b)$ may keep attaching itself to a successor chain around $c$ or $c'$.    
			
			We need a better strategy.  Instead of trying to define $a = f_{s+1}(b)$ immediately, we identify the first (relative to the standard ordering on pairs of the universe $\omega$ of $\mathcal{A}$) ``buffer pair'' $(z,z')$ such that $(c,z)$, $(z,z')$ and $(z',c')$ all appear infinite in $\mathcal{A}$.  Once we find such a $(z,z')$, then we search within $(z, z')$ for an element $a$ and a successor chain around it sufficient to match whatever one we may have created around $b$; we define $f_{s+1}(b) = a$.  Assuming the interval $(c, c')$ is correctly labeled as infinite, then, at some stage, we will settle on the first correct buffer pair $(z,z')$, i.e., one such that $(c,z)$, $(z,z')$ and $(z',c')$ all are really infinite in $\mathcal{A}$.  Then, applying the hypothesis about $\mathcal{A}$, we are guaranteed to find in $(z, z')$ an element $a$  with a finite interval around it large enough to correspond to whatever one we may have built around $b$ by this stage.  (Recall that, in general, when we map $b$ to $a$ for some requirement, we vow not to locate $b$ in a finite interval larger than the one we have seen around $a$.)  Following this procedure, we can eventually satisfy all requirements.                                     
		\end{proof}
	\end{proof}
	
	\subsection{Boolean algebras}
	\label{Boolean algebras}
	
	As we mentioned in the introduction, Downey and Jockusch \cite{DJ} showed that every low Boolean algebra has a computable copy.  In \cite{KS}, it is shown that for a low Boolean algebra $\mathcal{A}$, there is a computable copy $\mathcal{B}$ with a $\Delta^0_4$ isomorphism.  In unpublished work, the second author proved that this is best possible, in the sense that there is a low Boolean algebra with no $\Delta^0_3$ isomorphism taking $\mathcal{A}$ to a computable copy $\mathcal{B}$. 
	
	For every element $a$ in the Boolean algebra $\mathcal{B}$, we say that $a$ \emph{has size $n$} if it is the join of $n$ atoms of $\mathcal{B}$. If $a$ is not the join of finitely many atoms of $\mathcal{B}$, then we say that $a$ has \emph{infinite} size.  Here we consider Boolean algebras with no $1$-atoms, which means that every infinite element splits into two infinite elements.  
	To describe the $B_1$-type of a tuple $\bar{a}$ in $\mathcal{B}$, we consider the finite sub-algebra of $\mathcal{B}$ generated by $\bar{a}$.  Note that an atom in this finite sub-algebra is not necessarily an atom of $\mathcal{B}$.  It is easy to see that for a tuple $\bar{a}$ in $\mathcal{B}$, the $B_1$-type of $\bar{a}$ is uniquely determined by the sizes in $\mathcal{B}$ of the atoms in the finite sub-algebra generated by $\bar{a}$.  Thus, we can define a computable enumeration $R$ of all $B_1$-types realized in Boolean algebras so that from the index $i$ of the $B_1$-type $R_i$ we can effectively obtain the sizes of the atoms in the sub-algebra generated by a tuple that satisfies this $B_1$-type.
	
	Let $p(u)$ be a $B_1$-type saying that $u$ is infinite.  We need to know which $B_1$-types $q(u,x)$ are generated by $p(u)$ and existential formulas.  We have two interesting cases.
	
	\bigskip
	\noindent
	\textbf{Case 1}:  Let $q(u,x)$ be the $B_1$-type extending $p(u)$ in which $x$ splits $u$ into one finite element, say of size $k$, and one infinite element.  Let $t(u)$ be a complete type saying that $u$ has infinite size, but there are fewer than $k$ atoms below it.  Clearly, $p(u)$ is consistent with $t(u)$, whereas $q(u,x)$ is not.  By Lemma~\ref{prop:existential-types}, it follows that $q(u,x)$ is not generated by $p(u)$ and existential formulas.
	
	\bigskip
	\noindent
	\textbf{Case 2}:  Let $q(u,x)$ be the $B_1$-type extending $p(u)$ in which $x$ splits $u$ into two elements of infinite sizes.  Then $q(u,x)$ is generated by $p(u)$ and the infinite set of existential formulas saying that there are at least $n$ distinct elements below $x$ and $u\setminus x$,
	for every $n$.
	
	\bigskip
	
	The proof of the following is then straightforward.
	
	\begin{lem}
		
		If $\mathcal{A}$ is a Boolean algebra with no $1$-atoms, then $\mathcal{A}$ is weakly $1$-saturated.
		
	\end{lem}
	
	\begin{prop}
		
		Suppose $\mathcal{A}$ is an infinite Boolean algebra with no $1$-atoms.  Then $\mathcal{A}$ admits strong jump inversion.  Moreover, if $\mathcal{A}$ is low over $X$, there is an $X$-computable copy $\mathcal{B}$ with an isomorphism that is $\Delta^0_3$ relative to~$X$.
		
	\end{prop}
	
	\begin{proof}
		
		We are assuming that $\mathcal{A}$ is low over $X$.  To show that there is an $X$-computable copy, it is enough to show the following.
		
		\begin{lem} 
			
			Let $\mathcal{A}$ be Boolean algebra with no $1$-atom.  If $\mathcal{A}$ is low over $X$, then $X'$ computes a 
			copy $\mathcal{B}$ with an $R$-labeling.  Moreover, there is an isomorphism $f$ from $\mathcal{B}$ to $\mathcal{A}$ that is $\Delta^0_3$ relative to $X$.
			
		\end{lem}
		
		\begin{proof}
			
			For simplicity, we suppose that $\mathcal{A}$ is low, and our entire construction uses a $\Delta^0_2$ oracle.  For notational convenience, when we write $\bar{a} \in \mathcal{A}$ or $\bar{b} \in \mathcal{B}$, we identify the tuple with the finite sub-algebra (of $\mathcal{A}$ or $\mathcal{B}$) determined by the tuple.  Since $\mathcal{A}$ is low, the atom relation on $\mathcal{A}$ is $\Delta^0_2$.  Since we will guess (using the $\Delta^0_2$ oracle) that an element of $\mathcal{A}$ is finite iff we recognize it as the join of finitely many atoms of $\mathcal{A}$, any such guess is correct.  Now at a particular stage $s$, our guess may incorrectly assign an $R$-label of infinite to a finite element $a$ of $\mathcal{A}$; however, there will be a stage $t$ where we correctly guess the $R$-label of $a$ from that stage onward.  For any truly infinite element $a$, we guess the $R$-label correctly at all stages.
			
			We must computably (relative to $\Delta^0_2$) construct $\mathcal{B}$ with an $R$-labeling and an isomorphism $f$ between $\mathcal{B}$ and $\mathcal{A}$ that is correct in the limit, so that 
			$f$ is~$\Delta^0_3$.  
			
			As usual, we have the following requirements.
			
			\bigskip
			\noindent
			\textbf{$R_{2a}$}:  $a\in ran(f)$
			
			\bigskip
			\noindent
			\textbf{$R_{2b+1}$}:  $b\in dom(f)$
			
			\bigskip
			
			At stage $s = 0$, we define $f(0_{\mathcal{B}}) = 0_{\mathcal{A}}$ and $f(1_{\mathcal{B}}) = 1_{\mathcal{A}}$; this will never change.  We guess that $1_{\mathcal{A}}$ is labeled with $\infty$ (this will never be wrong), and we label $1_\mathcal{B}$ with $\infty$.  
			
			Assume that by the end of stage $s$ we have defined $\bar{b} \in \mathcal{B}$ with $R$-labels and $f_s: \bar{d} \rightarrow \bar{c}$, where $\bar{d}$ is a subsequence of $\bar{b}$, so that the following hold:
			
			\begin{enumerate}
				
				\item the finite algebras $\bar{d}$ and $\bar{c}$ agree;
				
				\item if $f_s(d) = c$, then the $R$-label on $d$ matches the stage $s$ approximation of the $R$-label on $c$;
				
				\item if $f_s(d) = c$, and the finite $R$-labels among those we have assigned to $\mathcal{B}_s$ imply that there are at least $k$ atoms (of $\mathcal{B}$) below $d$, then by stage $s$, we have seen at least $k$ atoms below $c$. 
				
			\end{enumerate}  
			
			Stage $(s+1)$ approximations of $R$-labelings of $\mathcal{A}$ may reveal that an element in $\mathcal{A}$ with stage $s$ approximate $R$-label $\infty$ actually is finite.  So in defining $f_{s+1}$, we first determine the largest initial segment of $f_s$ (in terms of priority requirements) that can be preserved.  Consider the highest priority requirement that now must be satisfied.
			
			Suppose the next requirement to be satisfied is to put $a$ into $ran(f)$.  The element $a$ splits each atom $\alpha$ of the subalgebra $\bar{c}$ into $\alpha_1$ and $\alpha_2$, each of which has a stage $s+1$ approximation of its $R$-label.  If $f_{s}(\beta) = \alpha$, then $\beta$ can be split---using the other elements of $\bar{b}$ or introducing new elements into $\mathcal{B}$ if necessary---into $\beta_1$ and $\beta_2$ so that if we extend $f_{s}$ by defining $f_{s+1}(\beta_1) = \alpha_1$ and  $f_{s+1}(\beta_2) = \alpha_2$, then properties (1) - (3) above are maintained and $R_{2a}$ is satisfied.  
			
			Suppose the next requirement to be satisfied is to put $b$ into $dom(f)$.  If $b$ has not yet appeared among $\bar{b}$, then simply extend $f_{s}$ to include $b$ in any way consistent with what we've defined so far about $\bar{b}$ and consistent with conditions (1)-(3) above, and define the $R$-labels on the elements of $\bar{b}, b$ accordingly.  Otherwise, $b$ splits each atom $\beta$ of the subalgebra $\bar{d}$ into $\beta_1$ and $\beta_2$, each of which has an $R$-label that must be preserved.  The only interesting case is when $\beta_1, \beta_2$ both have $R$-label $\infty$.  By conditions (1) and (2) above, $\alpha$, the atom in $\bar{c}$ corresponding to $\beta$, has a current approximate $R$-label $\infty$.  Because $\mathcal{A}$ contains no $1$-atom, we ``look ahead'' if necessary, either to discover that this $R$-label on $\alpha$ is incorrect, or to find the least (in terms of the universe $\omega$) element $\alpha_1$ below $\alpha$ so the approximate $R$-labels of both $\alpha_1$ and $\alpha - \alpha_1$ are $\infty$.  If we discover the former, then we must use a smaller  initial segment of $f_{s}$ and start over to satisfy a higher priority requirement.  Otherwise, we are almost ready to meet the requirement $R_{2b+1}$.  If the $R$-labels in $\bar{b}$ imply that there are at least $k_1$ atoms (of $\mathcal{B}$) below $\beta_1$ and at least $k_2$ atoms (of $\mathcal{B}$) below $\beta_2$, then by property (3) above, we have seen at least $k_1 + k_2$ atoms (of $\mathcal{A}$) below $\alpha$.  Consider the element $\alpha_1 \pm$ finitely many atoms below $\alpha$ so that this new element $\alpha_1'$ has at least $k_1$ atoms below it, and $\alpha - \alpha_1'$ has at least $k_2$ atoms below it.  Extend $f_{s}$ by defining $f_{s+1}(\beta_1) = \alpha_1'$ and  $f_{s+1}(\beta_2) = \alpha - \alpha_1'$.  Then properties (1) - (3) above are maintained and $R_{2b+1}$ is satisfied.   
		\end{proof}
	\end{proof}   
	
	\subsection{Trees}
	\label{trees}
	
	We consider some special classes of subtrees of $\omega^{<\omega}$.  Our trees grow downward.  The top node is $\emptyset$.  For the language of trees, we use the predecessor function, where $\emptyset$---the root---is its own predecessor.
	We consider two special classes of trees.  The first is very simple.

	\begin{prop}
		\label{simple.trees}
		
		Suppose $\mathcal{A}$ is a tree such that the top node is infinite (i.e., it has infinitely many successors), and each infinite node has only finitely many successors that are terminal, with the rest all infinite.  Then $\mathcal{A}$ admits strong jump inversion.
		
	\end{prop}
	
	\begin{proof}  
		
		The $B_1$-type of a tuple $\bar{a}$ is determined by the subtree generated by $\bar{a}$ and labels ``infinite'' or ``terminal'' on the nodes, in particular, on the $a_i$.  We have a computable enumeration of all possible labeled finite subtrees of trees of this kind.  From this, we get a computable enumeration $R$ of the $B_1$-types.  Suppose that 
		$\mathcal{A}$ is low.  Then there is a $\Delta^0_2$ $R$-labeling of $\mathcal{A}$.  
		
		\bigskip
		\noindent
		\textbf{Weak $1$-saturation}.  Take $\bar{a}$ in $\mathcal{A}$.  Consider a possible $B_1$-type $p(\bar{a},x)$, generated by formulas true of $\bar{a}$ and existential formulas.  The type may locate $x$ in the subtree generated by $\bar{a}$.  Then the type is realized.  The type may locate $x$ properly below some infinite $a_i$, or at some level not below any $a_i$.  Again the type is realized by a new infinite element.
		
		\bigskip
		
		By Theorem~\ref{simple.general}, we get a computable copy of $\mathcal{A}$.
	\end{proof}
	
	The second class of trees is a bit more complicated.  We use some definitions and notation.  If $T$ is a sub-tree of $\omega^{<\omega}$, and $a\in T$, we write $T_a$ for the tree consisting of $a$ and all nodes below.      
	
	\begin{defn}
		
		For nodes $a$ in a fixed tree $T$,   
		
		\begin{enumerate}
			
			\item  we say that $a$ is \emph{finite} if $T_a$ is finite,
			
			\item  we say that $a$ is \emph{infinite} if $T_a$ is infinite.  (For the trees we consider below, if $a$ is infinite, we will require not only that $T_a$ is infinite, but also that $a$ has infinitely many successors, so we will have agreement with the definition we used in Proposition~\ref{simple.trees}.)  
			
		\end{enumerate}
		
	\end{defn}
	
	
	
	
	
	
	
	
	
	\noindent
	\textbf{Notation}.  Let $a$ be finite, with $T_a$ the subtree below $a$.  Let $T^1_a$ be a possible re-labeling of the nodes in $T_a$ in which the nodes in a subtree are labeled $\infty$.  We write $(T^1_a)^*$ for the infinite tree that results from extending the labeled tree $T^1_a$ so that all new nodes in $(T^1_a)^*$ are labeled $\infty$, and each node labeled $\infty$ has infinitely many successors labeled $\infty$. (No finite node in $T^1_a$ acquires successors in $(T^1_a)^*$.)       
	
	\bigskip
	
	Here is the result for the second class of trees.
	
	\begin{prop}
		\label{prop1.9}
		
		Suppose $T$ is a subtree of $\omega^{<\omega}$ such that the top node is infinite, and for any infinite node $a$, there are only finitely many finite successors.  Suppose also that for any infinite node $a$, for any finite successor $b$, if $T^1_b$ is a possible re-labeling of $T_b$ making all nodes in a certain subtree infinite, then there are infinitely many successors $b_n$ of $a$ such that \\
		$T_{b_n}\cong (T^1_b)^*$.  Then $T$ admits strong jump inversion.
		
	\end{prop}
	
	\begin{proof}
		
		For simplicity, we suppose that $T$ is low, and we apply Theorem~\ref{simple.general} to produce a computable copy.
		For a tuple $\bar{a}$ in $T$, the $B_1$-type of $\bar{a}$ is determined by the subtree generated by $\bar{a}$ and formulas saying, for an element $a$ of this subtree that it is infinite, or that it is finite with a specific finite tree $T_a$.  
		We can show that $T$ is weakly $1$-saturated.  Consider a $B_1$-type for $\bar{a},x$, generated by $B_1$-formulas true of $\bar{a}$ and existential formulas.  The type may put $x$ in the subtree generated by $\bar{a}$, or in one of the trees $T_{a_i}$, where $a_i$ (in the subtree) is finite.  In either of these cases, the type is realized.  Or, the type may put $x$ below some infinite $a_i$ (in the subtree).  Again, the type is realized, since there is a copy of $\omega^{<\omega}$ below $a_i$.  This shows that $\mathcal{A}$ is weakly $1$-saturated.  
		
		We have a computable enumeration of the possible finite labeled subtrees, and, hence, of the $B_1$-types realized in trees of this kind.  Let $R$ be this computable enumeration of $B_1$-types.  To apply Theorem~\ref{simple.general}, we need the following.       
		
		\begin{lem}
			
			There is a copy $\mathcal{B}$ of $T$ with a $\Delta^0_2$ $R$-labeling.  
			
		\end{lem}
		
		\begin{proof}
			
			We build a $\Delta^0_2$ copy $\mathcal{B}$ of $T$ with nodes labeled as infinite, or with a specific finite tree below.  We suppose that the $\omega$-list of elements of $T$ has the feature that the top element comes first, and any other element comes after its predecessor.  This condition will also hold for the copy $\mathcal{B}$.  For $\mathcal{B}$, we label the top node $\infty$.  Having built a finite labeled subtree of $\mathcal{B}$, and determined a tentative partial isomorphism $f$ from this to a subtree of $T$, we may find that some first node $b$ labeled $\infty$ in $\mathcal{B}$ is mapped to a node $a$ in $T$ such that $T_a$ is actually finite.  The predecessor of $b$, say $b'$, is labeled $\infty$, and we may still believe that the predecessor $a'$ of $a$ in $T$ has an infinite tree below.  In our $\mathcal{B}$, we vow to add no more terminal nodes to $\mathcal{B}_b$ and we look for a successor $a''$ of $a$ with the appropriate $T_{a''}$.  At a given stage, we take the first $a''$ that seems to work.  Our first guess may not be correct---we may eventually see an unwanted finite node in $T_{a''}$.  However, because of the structural properties we are assuming about $T$, we will eventually find a good $a''$, with $T_{a''}$ matching our $\mathcal{B}_b$.          
		\end{proof}
		
		Applying Theorem~\ref{simple.general}, we get a computable copy of $T$.    
	\end{proof}

	\subsection{Models of a theory with few $B_1$-types}
	\label{small}
	
	Lerman and Schmerl \cite{LS} gave conditions under which an $\aleph_0$-categorical theory $T$ has a computable model.  They assumed that the theory is arithmetical and $T\cap\Sigma_{n+1}$ is $\Sigma^0_n$ for each $n$.  In \cite{K}, the assumption that $T$ is arithmetical is dropped, and, instead, it is assumed that $T\cap\Sigma_{n+1}$ is $\Sigma^0_n$ uniformly in $n$.  The proof in \cite{LS} gives the following.
	
	\begin{thm} [Lerman-Schmerl]
		
		Let $T$ be an $\aleph_0$-categorical theory that is $\Delta^0_N$ and suppose that for all $1\leq n < N$, $T\cap\Sigma_{n+1}$ is $\Sigma^0_n$.  Then $T$ has a computable model.     
		
	\end{thm}
	
	To prove this, Lerman and Schmerl showed the following.
	
	\begin{lem}
		
		For any $n < N$, if $\mathcal{A}$ is a model whose $B_{n+1}$-diagram is computable in $X'$, and $T\cap\Sigma_{n+2}$ is $\Sigma^0_1$ in $X$, then there is a model $\mathcal{B}$ whose $B_n$-diagram is computable in $X$. 
		
	\end{lem}
	
	Let $T$ be as in the Lerman-Schmerl Theorem.  Let $\mathcal{A}$ be a model of $T$ that is low over $X$.  Then the $\Sigma_1$ diagram of $\mathcal{A}$ is computable in $X'$.  Of course, 
	$T\cap \Sigma_2$ is $\Sigma^0_1$, so it is $\Sigma^0_1$ relative to $X$.  The lemma implies that $\mathcal{A}$ has an $X$-computable copy.  In fact, we get the following.  
	
	\begin{thm}  
		\label{LS}
		
		Let $T$ be an elementary first order theory, in a computable language, such that $T\cap\Sigma_2$ is $\Sigma^0_1$.  Suppose that for each tuple of variables $\bar{x}$, there are only finitely many $B_1$-types in variables $\bar{x}$ consistent with $T$.  Then every model $\mathcal{A}$ admits strong jump inversion.  Moreover, if $\mathcal{A}$ is low over $X$, then there is an $X$-computable copy $\mathcal{B}$ with an isomorphism that is $\Delta^0_2$ relative to $X$.     
		
	\end{thm}
	
	\begin{proof}   
		
		First, we show that there is a computable enumeration of all the $B_1$-types.  Uniformly in each tuple of variables $\bar{x}$, we build a c.e.\ tree $P_{\bar{x}}$ whose paths represent the $B_1$-types in $\bar{x}$.  We have a computable enumeration of $B_1$-formulas $(\varphi_n(\bar{x})_{n\in\omega}$.  At level $n$, the nodes $\sigma$ in $P_{\bar{x}}$ represent the different finite sequences of formulas $\pm\varphi_k$ (in the appropriate tuple of variables), for $k < n$, that we see to be consistent with $T$, using the fact that $T\cap\Sigma_2$ is c.e.  Note that each node $\sigma\in P_{\bar{x}}$ extends to a path.  Also, $P_{\bar{x}}$ has only finitely many paths.  We may suppose, running the enumeration of $T\cap\Sigma_2$ ahead, if necessary, that at step $s$, for the first $s$ tuples of variables $\bar{x}$, the terminal nodes in our approximation to $P_{\bar{x}}$ all have length $s$.            
		
		We use all of these trees together to define the enumeration $R$.  At stage $s$, we have assigned indices to the currently terminal nodes $\sigma$ in $P_{\bar{x}}$ for the first $s$ tuples of variables $\bar{x}$.  For a given node $\sigma$, assigned index $i$, we will have put into $R_i$ the formulas $\pm\varphi$ corresponding to this node $\sigma$.  At stage $0$, we assign the index $0$ to the top node of $P_\emptyset$.  At stage $s+1$, for each of the first $s$ tuples of variables $\bar{x}$, each node $\sigma$ of length $s$ in $P_{\bar{x}}$ has at least one extension of length $s+1$.  We give the index of $\sigma$ to one such $\tau$.  There may be further extensions of $\sigma$ or other old nodes, and we give these new indices.  In addition, for the $(s+1)^{st}$ tuple of variables $\bar{x}$, we assign indices to the terminal nodes of the stage $s+1$ approximation.  For the indices $i$ assigned by stage $s+1$ to nodes $\sigma$ of tree $P_{\bar{x}}$, we put into $R_i$ all of the formulas corresponding to $\sigma$.  This process yields the desired computable enumeration of the $B_1$-types consistent with $T$.                  
		
		Next, we show that $\mathcal{A}$ is weakly $1$-saturated.  Suppose $q(\bar{u},x)$ is a $B_1$-type (consistent, of course) generated by formulas true of $\bar{a}$ and existential formulas $\varphi(\bar{u},x)$.  Since $q(\bar{u},x)$ is isolated, it is principal, with a generating formula $\gamma(\bar{u},x)$, of the form $\rho(\bar{u})\ \&\ \chi(\bar{u},x)$, where $\rho(\bar{u})$ is in the $B_1$-type of $\bar{a}$, and $\chi(\bar{u},x)$ is a finite conjunction of existential formulas. $B_1$ type of $\bar{a}$ includes the formula $(\exists x)\chi(\bar{u},x)$. We have $(\exists x)\chi(\bar{u},x)$ true of $\bar{a}$ in $\mathcal{A}$, so the type is realized.    
		
		\begin{lem}
			
			If $\mathcal{A}$ is low over $X$, then there is an $R$-labeling of $\mathcal{A}$ that is $\Delta^0_2$ relative to $X$.  
			
		\end{lem}
		
		\begin{proof}
			
			For simplicity, we suppose $\mathcal{A}$ is low.  For a tuple of variables $\bar{x}$, $\Delta^0_2$ can find generating formulas for all of the $B_1$-types.  Then $\Delta^0_2$ can check which generating formula is true of a given tuple of elements $\bar{a}$.  Then we have a $\Delta^0_2$ $R$-labeling.    
		\end{proof}
		
		Finally, we apply Theorem~\ref{simple.general} to get an $X$-computable copy $\mathcal{B}$ of $\mathcal{A}$ with an isomorphism from $\mathcal{B}$ to $\mathcal{A}$ that is $\Delta^0_2$ relative to $X$.           
	\end{proof}
	
	\noindent
	\textbf{Note}:  There are non-$\aleph_0$-categorical theories satisfying the conditions of Theorem~\ref{LS}.
	
	\begin{proof} 
		
		We write $\Theta$ for the ordering of type $\eta + 2 + \eta$.  In \cite{DK}, it was shown that for any linear ordering 
		$\mathcal{A}$, $\Theta\cdot\mathcal{A}$ has a computable copy iff $\mathcal{A}$ has a $\Delta^0_2$ copy.
		Let $T_1$ be a complete theory of linear orderings that is not $\aleph_0$-categorical.  Let $T$ be the complete theory whose models are exactly the orderings of the form $\Theta\cdot\mathcal{A}$, where $\mathcal{A}$ is a model of $T_1$.  The theory $T$ has a sentence saying that every element lies on an interval of type $\Theta$.  In addition, there are axioms guaranteeing that the restriction of our ordering to the set of elements that are the first in a successor pair satisfies all sentences $\varphi$ in~$T_1$.  
		
		We note that the $B_1$-types realized in models of $T$ come from partitions into intervals of size $0$ or $\infty$, with no two adjacent intervals of size $0$.  These are principal, so they are realized in all models of $T$.  We note that if we replace $T_1$ by some other theory $S_1$ of infinite linear orderings, and form $S$ in the same way, then the $B_1$-types realized in any and all models of $S$ would be the same.  Therefore, the $\Sigma_2$ theories are the same.  If $S_1$ is decidable, then so is $S$.  Thus, whether or not $T_1$ is decidable, $T\cap\Sigma_2$ is decidable.  We chose $T_1$ not $\aleph_0$-categorical, so $T$ is also not $\aleph_0$-categorical.   
	\end{proof}
	
	\subsection{Differentially closed fields}
	\label{DCF_0} 
	
	\subsubsection{$DF_0$}
	
	A \emph{differential field} is a field with one or more derivations satisfying the following familiar rules:
	\begin{enumerate}
		
		\item $\delta(u+v) = \delta(u) + \delta(v)$, and 
		
		\item  $\delta(u\cdot v) = u\cdot\delta(v) + \delta(u)\cdot v$.  
		
	\end{enumerate}
	We consider differential fields of characteristic $0$, and with a single derivation~$\delta$.
	
	Trivially, $\mathbb{Q}$ is a differential field, under the derivation that takes all elements to $0$.
	If $a$ is an element of a differential field $K$, then $a$ \emph{generates} a differential field $F\subseteq K$, where the elements of $F$ are gotten from $a$ by closing under addition, multiplication, subtraction, division, and derivation.         
	
	\subsubsection{$DCF_0$}
	
	Roughly speaking, a \emph{differentially closed field} is a differential field in which differential polynomials have roots, where a differential polynomial is a polynomial $p(x)$ in $x$ and its various derivatives.  We write $DCF_0$ for the theory of differentially closed fields (of characteristic $0$, with a single derivation).  A.\ Robinson showed that the theory $DCF_0$ admits elimination of quantifiers.  L.\ Blum, in her thesis, gave a nice computable set of axioms, showing that the theory is decidable.  Thus, the elimination of quantifiers is effective.  Blum also showed that $DCF_0$ is $\omega$-stable.  Then general model-theoretic results imply the existence and uniqueness of prime models over an arbitrary set.  The existence and uniqueness of differential closures were not proved by algebraic methods---they really used the model theoretic results.  For a discussion of differentially closed fields, emphasizing Blum's results, see Sacks \cite{Sacks}. 
	
	\subsubsection{Differential polynomials}  
	
	We consider \emph{differential polynomials} $p(x)$ in a single variable $x$.  A differential polynomial $p(x)$, over a differential field $K$, may be thought of as an algebraic polynomial in $K[x,\delta(x),\delta^{(2)}(x),\ldots,\delta^{(n)}(x)]$, for some $n$.  We write $K\langle x\rangle$ for the set of differential polynomials over $K$.  Initially, we let $K$ be $\mathbb{Q}$, where $\delta(q) = 0$ for all $q\in\mathbb{Q}$.  Later, $K$ will be a finitely generated extension of $\mathbb{Q}$.  Differential fields satisfy the quotient rule---this is easy to prove from the product rule.  From this, it follows that if $a$ is an element of a differential field extending $K$, and $F$ is the differential subfield generated over $K$ by $a$, then each element of $F$ can be expressed in the form $\frac{p(a)}{q(a)}$, where $p(x),q(x)\in K\langle x\rangle$.   
	
	\begin{defn} [Order]
		
		For $p(x)\in K\langle x\rangle$, the \emph{order} is the greatest $n$ such that $\delta^{(n)}(x)$ appears non-trivially in $p(x)$.  There are some special cases.  An algebraic polynomial in $x$ (with no derivatives) has order $0$.  The $0$ polynomial has order $\infty$.  
		
	\end{defn}
	
	\begin{defn} [degree, rank, order of ranks]
		
		For $p(x)\in K\langle x\rangle$ of finite order $n$, the \emph{degree} of $p(x)$ is the highest power $k$ of $\delta^{(n)}(x)$ that appears.  The \emph{rank} of $p(x)$ is the ordered pair $(n,k)$, where $n$ is the order and $k$ is the degree.  We order the possible ranks of differential polynomials lexicographically.    
		
	\end{defn}  
	
	\begin{defn}
		
		A differential polynomial $p(x)\in K\langle x\rangle$ of order $n$ is said to be \emph{irreducible} if it is irreducible when considered as an algebraic polynomial in $K[x,\delta(x),\ldots,\delta^{(n)}(x)]$ (think of $x$ and its derivatives as indeterminates).  We count the $0$ polynomial as irreducible. 
		
	\end{defn}     
	
	\subsubsection{Blum's axioms for $DCF_0$}
	
	Blum's axioms say that a differentially closed field (of characteristic $0$ and with a single derivation), is a differential field $K$ such that 
	\begin{enumerate}
		
		\item  for any pair of differential polynomials $p(x),q(x)\in K\langle x\rangle$ such that the order of $q(x)$ is less than that of $p(x)$, there is some $x$ satisfying $p(x) = 0$ and $q(x)\not= 0$,  
		
		\item  if $p(x)$ has order $0$, then $p(x)$ has a root.  
		
	\end{enumerate}
	The axioms of form (2) say that $K$ is algebraically closed.  
	
	\subsubsection{Types}    
	
	We want to understand the types, in any number of variables, realized in models of $DCF_0$.  For a single variable $x$, each type over $\emptyset$ is determined by an irreducible differential polynomial $p(x)\in\mathbb{Q}\langle x\rangle$.  If $p(x)\in \mathbb{Q}\langle x\rangle$ is irreducible of order $n$, then the corresponding type consists of formulas provable from the axioms of $DCF_0$, the formula $p(x) = 0$ and further formulas $q(x)\not= 0$, for $q(x)\in \mathbb{Q}\langle x\rangle$ of order less than $n$.  The formulas $q(x)\not= 0$, for $q(x)\in\mathbb{Q}\langle x\rangle$ of order less than $n$, say that 
	$x,\delta(x),\delta^{(2)}(x), \ldots, \delta^{(n-1)}(x)$ are algebraically independent over $\mathbb{Q}$.  We allow the case where $p(x)$ is the $0$ polynomial, which has order $\infty$.  In this case, the corresponding type $\lambda_p$ consists of the formulas provable from the axioms of $DCF_0$ and the formulas $q(x)\not= 0$ for $q(x)$ of all finite orders.     
	
	Similarly, for a differential field $K$, each type over $K$ (to be realized in some extension of $K$ to a model of $DCF_0$) is determined by an irreducible differential polynomial $p(x)\in K \langle x\rangle$.  If $p(x)$ is irreducible of order $n$, the corresponding type $\lambda_{K,p}$ consists of formulas provable from the axioms of $DCF_0$, the atomic diagram of $K$, the formula $p(x) = 0$, and further formulas $q(x)\not= 0$, for $q(x)$ of order less than $n$.  The formulas $q(x)\not= 0$, taken together, say that 
	$x,\delta(x),\ldots,\delta^{(n-1)}(x)$ are algebraically independent over $K$. 
	
	A proof of the following result can be found in Sacks \cite{Sacks}, pp.\ 297-298.   
	
	\begin{prop}\
		
		\begin{enumerate}
			
			\item  If $p(x)\in \mathbb{Q}\langle x\rangle$ is irreducible, the corresponding type $\lambda_p$ is complete over $\emptyset$.  Moreover, all types over $\emptyset$ (in the variable $x$) have this form.
			
			\item  For a differential field $K$, if $p(x)\in K\langle x\rangle$ is irreducible, then $\lambda_{K,p}$ is a complete type over $K$, and all types over $K$ (in the variable $x$) have this form.  
			
		\end{enumerate}
		
	\end{prop}     
	
	Among the types in one variable (over $\emptyset$, or over $K$), there is a unique type, obtained from the $0$ polynomial, that is  \emph{differential transcendental}.  The other types, obtained from differential polynomials of finite rank, are \emph{differential algebraic}.  
	
	\subsubsection{Types in several variables}         
	
	In general, we can determine a type in variables $(x_1,\ldots,x_n)$ by giving the type of $x_1$ (over 
	$\emptyset$), the type of $x_2$ over $x_1$, the type of $x_3$ over $(x_1,x_2)$, and so on.  To describe a type in variables $(x_1,\ldots,x_n)$, we imagine a large differentially closed field $M$ and we consider various elements and differential subfields.  The type of $x_1$ is $\lambda_{p_1}$ for some irreducible $p_1\in \mathbb{Q}\langle x_1\rangle$.  Let $K_1$ be the differential subfield of $M$ generated by $x_1$ over $\mathbb{Q}$, where $x_1$ satisfies $\lambda_{p_1}$ in $M$.  The type of $x_2$ over $K_1$ is $\lambda_{K_1,p_2}$ for some irreducible $p_2\in K_1\langle x_2\rangle$.  Let $K_2$ be the differential field generated by $x_2$ over $K_1$.  In general, given $K_i$ generated by $x_1,\ldots,x_i$, the type of $x_{i+1}$ over $K_i$ is $\lambda_{K_i,p_{i+1}}$ for some irreducible 
	$p_{i+1}\in K_i\langle x_{i+1}\rangle$, and then $K_{i+1}$ is the differential subfield of $M$ generated by $x_{i+1}$ over $K_i$.      
	
	\subsubsection{Toward strong jump inversion}
	
	Marker and R.\ Miller \cite{MM} showed that all models of $DCF_0$ admit strong jump inversion.  Our goal in this subsection is to obtain this result using our Theorem~\ref{simple.general}. In the earlier applications of Theorem \ref{simple.general}, the structures satisfied the condition of effective type completion because they were weakly $1$-saturated. Among the countable models of $DCF_0$, only the saturated one is weakly $1$-saturated.  There are $2^{\aleph_0}$ non-isomorphic countable models.  (In fact, Marker and Miller gave a method for coding an arbitrary countable graph in a model of $DCF_0$.)  We will need to show effective type-completion in some other way.  There is a lemma in \cite{MM} that does exactly this.  Since we have effective quantifier elimination, we can work with quantifier-free types. Most of our effort goes into producing a computable enumeration $R$ of the quantifier-free types realized in models of $DCF_0$.  Once we have this, we can show easily that for any model $\mathcal{A}$, $D(\mathcal{A})'$ computes an $R$-labeling of $\mathcal{A}$.  This puts us in position to apply 
	Theorem~\ref{simple.general}. 
	
	\subsubsection{Computable enumeration of types} 
	
	It may at first seem that it should be easy to produce a computable enumeration of types.  After all, the theory $DCF_0$ is decidable and all types are computable. However, T.\ Millar \cite{Millar} gave an example of a decidable theory $T$, with all types computable, such that there is no computable enumeration of all types.  So, we have some work to do.  
	
	By quantifier elimination, we can pass effectively from a quantifier-free type $\lambda(\bar{x})$ to the complete type generated by $DCF_0\cup\lambda(\bar{x})$.  In what follows, we will enumerate quantifier-free types.  We will consider realizations of the quantifier-free types in differential fields $K$ that are not differentially closed, bearing in mind that a tuple realizing $\lambda(\bar{x})$ in $K$ will realize the corresponding complete type generated by $DCF_0\cup\lambda(\bar{x})$ in any extension of $K$ to a model of $DCF_0$.  
	
	We eventually give a uniform procedure that, for a given tuple of variables $\bar{x}$, yields an enumeration of the types in $\bar{x}$.  But first, we give a procedure for a single variable $x$ in order to elucidate the relevant issues before proceeding to the full procedure.  We determine a type $\lambda(x)$ corresponding to each differential polynomial 
	$p(x)\in \mathbb{Q}\langle x\rangle$, irreducible or not.  Let $(\varphi_s)_{s\in\omega}$ be a computable list of the atomic formulas in variable $x$, in order of G\"{o}del number.  At each stage, we will have put into $\lambda(x)$ finitely many formulas, always checking consistency with $DCF_0$.    
	
	At stage $0$, we put into the type $\lambda(x)$ just the formula $p(x) = 0$, assuming that this is consistent.  We also determine the order of $p(x)$---we can do this just by inspection.  At stage $s$, we will decide $\varphi_s$, putting it or its negation into $\lambda(x)$.  If $p(x)$ is irreducible, there will be a proof of exactly one of $\varphi_s$, $\neg{\varphi_s}$ from $DCF_0$, $p(x) = 0$, and the formulas $q(x)\not= 0$, for $q(x)\in \mathbb{Q}\langle x\rangle$ of order less than that of $p(x)$.  So, we search for a proof.  Being reducible is c.e., and if $p(x)$ is reducible, we will eventually see this.     
	
	At stage $s$, we search until we either find a proof of $\pm\varphi_s$ or discover that $p(x)$ is reducible.  If we find a proof of $\varphi_s$ (or $\neg{\varphi_s}$), then we add this formula to our type, provided that it is consistent to do so.  If we find that $p(x)$ is reducible, then we just decide $\varphi_s$ so as to maintain consistency with $DCF_0$.  The procedure we have just described gives a type $\lambda$ corresponding to each $p\in\mathbb{Q}\langle x\rangle$.  If $p$ is irreducible, then $\lambda = \lambda_p$.  Thus, by considering all $p\in \mathbb{Q}\langle x\rangle$, we get all types in the variable $x$.  
	
	A type in one variable corresponded to a differential polynomial $p(x)$ over $\mathbb{Q}$.  Intuitively, we'd like to enumerate types in $n$ variables using all $n$-tuple of polynomials, according to the pattern described in subsection 3.5.6.  Unfortunately, since the fields themselves depend on the polynomials in the tuple, it is not even clear if a potential polynomial would make sense; one of its coefficients might actually be undefined.  Therefore, our enumeration construction takes these obstacles into account with a more formal approach.  A type in $n$ variables will correspond to an $n$-tuple of formal differential polynomials $p_1(x_1),\ldots,p_n(x_n)$.  Here $p_1(x_1)$ is an actual differential polynomial with coefficients in $\mathbb{Q}$.  For $i \geq 1$, $p_{i+1}(x_{i+1})$ looks like a differential polynomial, but the coefficients come from a set $K_i^F$ of formal names for possible elements of a differential field generated by elements $x_1,\ldots,x_i$.  We say more about these formal names below.  We define the sets $K_i^F$ and $K_i^F\langle x_{i+1}\rangle$ by induction on $i$.   
	
	The many lemmas below allow us to prove Proposition \ref{enumeration}, the computable enumeration of types, from the basic definitions and results in \cite{Sacks}. In personal correspondence, Marker and Pillay described less elementary but more efficient ways of deriving this result using more modern machinery about differentially closed fields.     
	
	\begin{defn}\
		
		\begin{enumerate}
			
			\item  $K_0^F = \mathbb{Q}$, and $K_0^F\langle x_1\rangle = \mathbb{Q}\langle x_1\rangle$,
			
			\item  $K_i^F\langle x_{i+1}\rangle$ is the set of formal expressions that look like differential polynomials in the variable $x_{i+1}$ but have coefficients in $K_i^F$ as opposed to a well-defined differential field,  
			
			\item  $K_{i+1}^F$ consists of the expressions $\frac{r(x_{i+1})}{s(x_{i+1})}$, where $r,s\in K_i^F\langle x_{i+1}\rangle$.  
			
		\end{enumerate}
		
	\end{defn}    
	
	\begin{lem}
		\label{lem3.6}
		
		Uniformly in $n$, we can enumerate the $n$-tuples $p_1(x_1),\ldots,p_n(x_n)$, where $p_{i+1}(x_{i+1})\in K_i^F\langle x_{i+1}\rangle$. 
		
	\end{lem} 
	
	\begin{proof}
		
		The set $K_0^F$ is a fixed computable set with computable index, and there is a uniform, effective procedure to construct $K_i^F \langle x_{i+1} \rangle$ from $K_i^F$ and $K_{i+1}^F$ from $K_i^F \langle x_{i+1} \rangle$.  Therefore, there is a single, computable function that gives computable indices for all of these sets.  Then there is computable function that, given $n$, finds a computable index of $K_0^F \langle x_1 \rangle \times K_1^F \langle x_2 \rangle \times \cdots \times K_{n-1}^F \langle x_n \rangle$. 
	\end{proof}
	
	Given an $n$-tuple of formal differential polynomials $p_1,\ldots,p_n$ as above, we will obtain a type 
	$\lambda(x_1,\ldots,x_n)$ by producing a sequence of differential fields $K_0,\ldots,K_n$, where $K_0 = \mathbb{Q}$, and $K_{i+1}$ is generated over $K_i$ by an element $x_{i+1}$ satisfying a chosen type $\lambda_{i+1}$ that depends on $p_{i+1}$.  In the end, $K_n$ will be generated by $x_1,\ldots,x_n$, and $\lambda(x_1,\ldots,x_n)$ will be the type realized by $x_1,\ldots,x_n$ that generates $K_n$.  We give several lemmas.   
	
	\begin{lem}
		\label{lem3.7} 
		
		There is a uniform effective procedure that, given a differential field $K$ and a type $\lambda(x)$ over $K$, yields a differential field $K'\supseteq K$ that is generated over $K$ by an element $x$ realizing $\lambda$.
		
	\end{lem} 
	
	\begin{proof}  Uniformly in $K$, we construct a computable, formal set $N_K$ that consists of names of the form $\frac{r(x)}{s(x)}$, where $r(x), s(x) \in K \langle x \rangle$.  Next, we define the universe of $K'$ from $N_K$ and $\lambda(x)$ by induction:
		
		\begin{enumerate}
			\item at step 1, consider the first element of $N_K$, which is a formal expression of the form $\frac{r(x)}{s(x)}$.  We use $\lambda(x)$ to determine if $s(x) = 0$.  If so, then we do not include $\frac{r(x)}{s(x)}$ in the universe of $K'$; otherwise we do.
			
			\item at step $n + 1$, consider the $(n + 1)^{st}$ element of $N_K$, which is a formal expression of the form $\frac{r(x)}{s(x)}$, where $r(x), s(x) \in K \langle x \rangle$.  We use $\lambda(x)$ to determine if $s(x) = 0$.  If so, then we do not include $\frac{r(x)}{s(x)}$ in the universe of $K'$.  If not, then we use $\lambda(x)$ and simple ``cross multiplication'' to determine if $\frac{r(x)}{s(x)}$ is equal to any $\frac{r_1(x)}{s_1(x)}$ that we included in $K'$ an earlier step.  If so, then we do not include $\frac{r(x)}{s(x)}$ in the universe of $K'$; otherwise we do.
			
		\end{enumerate}
		
		\noindent Uniformly in $K$ and $\lambda(x)$, the above procedure computably enumerates the elements of the universe of $K'$ in order; therefore, the universe of $K'$ is uniformly computable in $K$ and $\lambda(x)$.  
		
		Finally, to define the constants and operations on the universe of $K'$, we first use $\lambda(x)$ to identify element in the universe $K'$ that is equal to $\frac{0_K}{1_K}$ and the element equal to $\frac{1_K}{1_K}$.  Next, to calculate a sum or product of two elements $\frac{r(x)}{s(x)}$ and $\frac{r_1(x)}{s_1(x)}$ in $K'$, or to calculate $\delta \left( \frac{r(x)}{s(x)} \right)$, we add, multiply, or differentiate formally, and then we use $\lambda(x)$ to determine what element in the universe of $K'$ the formal expression is equal to.  The definitions of these operations are uniformly computable from $K'$ and $\lambda$, and thus ultimately from $K$ and $\lambda$.  
	\end{proof}
	
	Given an actual differential field $K_i$, generated by elements $x_1,\ldots,x_i$, some names from $K_i^F$ have a definite value in $K_i$, while others do not.  Recall that the names are quotients.  We do not get a value if the denominator is $0$.  
	
	\begin{lem}
		\label{lem3.8}  
		
		There is a uniform effective procedure that, given a differential field $K_i$ generated by elements $x_1,\ldots,x_i$, and an element $f\in K^F_i$, determines whether $f$ makes sense, and if so, assigns to $f$ a definite value in $K_i$.
		
	\end{lem}
	
	\begin{proof}
		
		We first form a finite set $S$ of names such that $f\in S$, and if $g\in S\cap K^F_j$, for $0 < j\leq i$, and $h$ is a coefficient from the numerator or denominator of $g$, then $h\in S$.  We form $K_j$ for $0\leq j\leq i$.  We then proceed by induction on $j$ to determine for all $g\in S\cap K_j^F$, whether $g$ has value in $K_j$, and if so, to assign the value.  Then $f$ has a value iff all elements of $S$ have a value.    
	\end{proof}    
	
	\begin{lem}
		\label{lem3.9} 
		
		There is a uniform effective procedure that, given $p\in K^F_i\langle x_{i+1}\rangle$ and a differential field $K_i$ generated by elements $x_1,\ldots,x_i$, determines whether $p$ makes sense (i.e., whether the coefficients all have value in $K_i$), and if so, identifies $p$ with an element of $K_i\langle x_{i+1}\rangle$.     
		
	\end{lem} 
	
	\begin{proof} Given the $p\in K^F_i\langle x_{i+1}\rangle$, we simply identify its coefficients as elements of $K^F_i$.  Then we apply the previous lemma to each of the co-efficients individually.  If all of the co-efficients make sense, then we assign each of them a definite value in $K_i$ and then construct the corresponding element of $K_i\langle x_{i+1}\rangle$.  Otherwise, if at least one of the coefficients does not make sense, $p$ does not make sense. 
	\end{proof}
	
	\begin{lem}
		\label{lem3.10}
		
		There is a uniform effective procedure that, given a differential field $K$ and a differential polynomial $p(x)$ over $K$, enumerates the differential polynomials $q(x)$ of order lower than that of $p(x)$.
		
	\end{lem}
	
	\begin{proof}
		
		First, there is a computable procedure, uniform in $K$, that computes orders of $p(x) \in K\langle x \rangle$.  Namely, assuming $p(x)$ is written where formal ``like terms'' already are combined, then the procedure looks for the term with the highest derivative $\delta^{n}(x)$ appearing as a factor, where the coefficient in $K$ for at least one such term is non-zero.  Then, uniformly in $K$ and $p(x)$, there is an effective procedure that lists all algebraic polynomials in $K[x, \delta(x), \ldots, \delta^{n-1}(x)]$.
	\end{proof}
	
	\begin{lem}
		\label{lem3.11}
		
		There is a uniform effective procedure that, given a differential field $K$ and a differential polynomial $p(x)$ over $K$, enumerates the proofs of formulas $\varphi(x)$ (with parameters in $K$) from $DCF_0$, $D(K)$, $p(x) = 0$, and $q(x)\not= 0$, for $q$ of lower order.
		
	\end{lem}
	
	\begin{proof}
		
		By Lemma \ref{lem3.10}, we can enumerate the polynomials $q(x)$ of lower order, so we can enumerate the axioms to use in our proofs.  Then we can enumerate proofs from these axioms of formulas of the kind we are interested in.      
	\end{proof}
	
	In Lemma \ref{lem3.11}, we did not assume that $p(x)$ is irreducible.  So, the set of axioms may not generate a consistent, complete type over $K$.  
	
	\begin{lem}
		\label{lem3.12}
		
		There is a uniform effective procedure that, given a differential field $K$, enumerates the reducible differential polynomials $p(x)$ over $K$.
		
	\end{lem}
	
	\begin{proof}
		
		For a given $p(x)$ we enumerate $D(K)$, searching for a formula of form $r(x)\cdot s(x) = p(x)$, where $r(x)$ and $s(x)$ are differential polynomials over $K$, both non-constant.  The search halts iff $p(x)$ is reducible.           
	\end{proof}
	
	\begin{lem}
		\label{added1}
		
		Let $K$ be a differential field.  For any tuple $\bar{k}$ in $K$, $DCF_0$ together with the quantifier-free type of $\bar{k}$ generates a complete type that would be realized by $\bar{k}$ in any extension of $K$ to a model of $DCF_0$.
		
	\end{lem}
	
	\begin{proof}
		
		Let $K\subseteq M$, where $M$ is a differentially closed field.  By quantifier elimination, any formula true of $\bar{k}$ in $M$ is proved from $DCF_0$ and the quantifier-free formulas true of $\bar{k}$.    
	\end{proof}
	
	\begin{lem}
		\label{added2}
		
		There is a uniform effective procedure for determining, for a differential field $K$ and a formula $\varphi(\bar{k},x)$ (with parameters $\bar{k}$ in $K$), whether $\varphi(\bar{k},x)$ is consistent with $DCF_0\cup D(K)$.
		
	\end{lem}
	
	\begin{proof}
		
		Let $\gamma(\bar{k})$ be the quantifier-free type realized by $\bar{k}$ in $K$.  By Lemma~\ref{added1}, $DCF_0\cup\gamma(\bar{k})$ generates a complete type that would be realized by $\bar{k}$ in any extension of $K$ to a model of $DCF_0$.  Then $\varphi(\bar{k},x)$ is consistent with $DCF_0\cup D(K)$ iff $(\exists x)\varphi(\bar{k},x)$ is in this type.  
	\end{proof} 
	
	\begin{lem}
		\label{lem3.13}
		
		There is a uniform effective procedure that, given a differential field $K$ and $p(x)\in K\langle x\rangle$, enumerates a type $\lambda(x)$ for $x$ over $K$.  Moreover, if $p(x)$ is irreducible, then $\lambda(x) = \lambda_{K,p}$.    
		
	\end{lem}
	
	\begin{proof}
		
		We can determine the order of $p(x)$, just by inspection.  At each step, we will have put finitely many formulas into the type $\lambda(x)$, having checked consistency with $DCF_0\cup D(K)$ as in Lemma \ref{added2} (the parameters from $K$ that appear in the formulas form the relevant $\bar{k}$).   
		At step $0$, we put into $\lambda(x)$ the formula $p(x) = 0$, assuming that this is consistent.  We have a computable enumeration of the atomic formulas $\varphi_s(x)$ with parameters in $K$.  At step $s+1$, we decide $\varphi_s(x)$, adding $\varphi_s(x)$ or $\neg{\varphi_s(x)}$ to the type $\lambda(x)$. If we have already seen that $p(x)$ is reducible, then we add $\varphi_s(x)$ to the type if it is consistent to do so, and otherwise, we add $\neg{\varphi_s(x)}$.  Suppose that $p(x)$ appears to be irreducible.  Then we simultaneously search for the following:
		
		\begin{enumerate}
			
			\item  a proof of $\pm\varphi_s$ from $DCF_0\cup D(K)$, $p(x) = 0$, and formulas $q(x)\not= 0$ for $q$ of order less than that of $p$,
			
			\item  evidence that $p(x)$ is reducible over $K$.
			
		\end{enumerate}
		
		By Lemmas \ref{lem3.11} and \ref{lem3.12}, these are computable searches.  One of the searches will halt, since if $p(x)$ is irreducible, then the formulas in (1) above generate a complete type over $K$. If we find that $p(x)$ is reducible, then we proceed as above, adding $\pm\varphi_s(x)$ just to maintain consistency.  (We check consistency as in Lemma \ref{added2}.)  If we find a proof of $\varphi_s$, or $\neg{\varphi_s}$, then we add this formula to the type, provided that it is consistent to do so.  We take inconsistency as evidence that $p(x)$ is reducible, and we proceed as above.   
	\end{proof}
	
	\begin{prop}
		
		Uniformly in $n$, we can enumerate the types in $n$ variables.
		
	\end{prop}
	\label{prop3.14}
	
	\begin{proof}
		
		By Lemma \ref{lem3.6}, uniformly in $n$, we can enumerate the $n$-tuples of formal differential polynomials $p_1,\ldots,p_n$, where $p_1\in\mathbb{Q}\langle x_1\rangle$, $p_{i+1}\in K_i^F\langle x_{i+1}\rangle$.  The $j^{th}$ $n$-tuple of differential polynomials $p_1,\ldots,p_n$ will yield the $j^{th}$ type $\lambda(x_1,\ldots,x_n)$ in variables $x_1,\ldots,x_n$.  We describe $\lambda(x_1,\ldots,x_n)$ in terms of some differential fields $K_1,\ldots,K_n$ and types $\lambda_i(x_i)$ over $K_{i-1}$.  We note that $p_1$ is an actual differential polynomial over $K_0 = \mathbb{Q}$.  We apply Lemma \ref{lem3.13} to $p_1$ and $\mathbb{Q}$, to get a type $\lambda_1(x_1)$.  We apply Lemma \ref{lem3.7} to $\mathbb{Q}$ and $\lambda_1$ to get the differential field $K_1$ generated by $x_1$ realizing $\lambda_1$.  
		
		Now, $p_2(x_2)$ is only an element of $K^F_1\langle x_2\rangle$, where $K^F_1$ is not an actual differential field.  We apply Lemma \ref{lem3.9} to $K_1$ to determine whether $p_2(x_2)$ makes sense as a differential polynomial over $K_1$.  If not, then we generate a type $\lambda_2$ for $x_2$ over $K_1$ using $DCF_0\cup D(K_1)$ as follows.  We run through the atomic formulas $\varphi_s(x_2)$ (over $K_1$) in order, adding $\varphi_s$ if it is consistent to do so, and otherwise adding $\neg{\varphi_s}$.  We check consistency at each step as in Lemma \ref{added2}.  If $p_2$ makes sense as a differential polynomial over $K_1$, then we apply Lemma \ref{lem3.13} to get $\lambda_2$.  We then apply Lemma \ref{lem3.7} to get the differential field $K_2$ generated by $x_2$ realizing $\lambda_2$ over $K_1$.  
		
		In general, given $K_i$, we apply Lemma \ref{lem3.9} to determine whether $p_{i+1}$ makes sense as differential polynomial over $K_i$.  If not, then we generate a type $\lambda_{i+1}$, using $DCF_0\cup D(K_i)$.  If $p_{i+1}$ makes sense as a differential polynomial over $K_i$, then we apply Lemma \ref{lem3.13} to get a type $\lambda_{i+1}$ for $x_{i+1}$ over $K_i$.  From $K_i$ and $\lambda_{i+1}$, we get $K_{i+1}$ as in Lemma \ref{lem3.7}.  After finitely many steps, calculating computable indices for the differential fields $K_i$ and the types $\lambda_i$, we arrive at the differential field $K_n$.  This is generated over $\mathbb{Q}$ by the elements $x_1,\ldots,x_n$. The quantifier-free type we want is that realized by $x_1,\ldots,x_n$ in $K_n$. Of course, since $DCF_0$ has effective quantifier elimination, we then effectively compute the complete type realized by $x_1,\ldots,x_n$.    
	\end{proof} 
	
	As planned, we combine the enumerations of types in variables $x_1,\ldots,x_n$, for various $n$.   
	
	\begin{prop}
		\label{enumeration}
		
		There is a computable enumeration $R$ of all complete types realized in models of $DCF_0$.
		
	\end{prop} 
	
	Now, we can prove the result of Marker and Miller, using our Theorem \ref{simple.general}.
	
	\begin{prop}
		\label{diff}
		
		Every countable model of $DCF_0$ admits strong jump inversion.
		
	\end{prop}
	
	\begin{proof}
		
		By Proposition \ref{enumeration}, there is a computable enumeration $R$ of the complete types realized in models of $DCF_0$, and thus, of the $B_1$ types.  Thus, Condition (1) of Theorem \ref{simple.general} holds.  The following lemma shows that Condition (3) holds in the strong way.  
		
		\begin{lem}     
			
			Let $X$ be a subset of $\omega$, and let $\mathcal{A}$ be a model of $DCF_0$ that is low over $X$.  Then $X'$ computes an $R$-labeling of $\mathcal{A}$.  
			
		\end{lem}
		
		\begin{proof} [Proof of Lemma]
			
			Note that for each tuple $\bar{a}$, we have an $\mathcal{A}$-computable procedure for finding, at step $s$, the first index $i$ such that $R_i$ agrees with the type of $\bar{a}$ on the first $s$ quantifier-free formulas.  After some step $s$, this $i$ is the first index for the $B_1$-type of $\bar{a}$.  Thus, we have an $R$-labeling that is computable in $D(\mathcal{A})'$, and hence, in $X'$, since $\mathcal{A}$ is low over $X$.
		\end{proof}    
		
		We need to establish Condition (2), effective type completion.  There is a uniform effective procedure for computing, from a type $p(\bar{u})$ and a formula $\varphi(\bar{u},x)$, consistent with $p(\bar{u})$, a type $q(\bar{u},x)$ such that if $\bar{c}$ satisfies $p(\bar{u})$, then some $a$ satisfies $q(\bar{c},x)$.  Marker and Miller \cite{MM} needed this for the same reason we do.  It is Lemma 4.3 in their paper.  (The type $q(\bar{c},x)$ will be realized in the differential closure of $\bar{c}$.)  The conditions for Theorem \ref{simple.general} are all satisfied.  Therefore, $\mathcal{A}$ admits strong jump inversion. 
	\end{proof}
	
	\subsubsection{Decidable saturated model of $DCF_0$}
	
	In general, a structure $\mathcal{A}$ is computable if its atomic diagram is computable, and $\mathcal{A}$ is decidable if the complete diagram is computable.  By elimination of quantifiers, a model of $DCF_0$ is decidable iff it is computable.  Using Proposition \ref{enumeration}, we can show that the countable saturated model of $DCF_0$ has a decidable copy.  We need the following result from Morley \cite{Morley}.
	
	\begin{thm}
		\label{decidable}
		
		Let $T$ be a countable complete elementary first order theory for a computable language.  Then the following are equivalent:
		
		\begin{enumerate}
			
			\item  $T$ has a decidable saturated model,
			
			\item  there is a computable enumeration of all types realized in models of~$T$.
			
		\end{enumerate}
		
	\end{thm}

	Using Theorem \ref{decidable} and Proposition \ref{enumeration}, we get the following.        
	
	\begin{cor}
		
		The saturated model of $DCF_0$ has a decidable copy.  
		
	\end{cor}

\end{document}